\documentclass[10pt,reqno]{amsart}

\usepackage[utf8]{inputenc}
\usepackage[T1]{fontenc}
\usepackage[english]{babel}

\usepackage{microtype}
\usepackage{amsmath,amssymb,amsthm}
\usepackage{amsthm, thmtools, thm-restate}
\usepackage{mathtools} 
\usepackage{bm}        
\usepackage{enumitem}

\usepackage{graphicx}  
\usepackage{caption}
\usepackage{subcaption}
\usepackage{tikz}      
\usepackage{mathrsfs}  
\usepackage{booktabs}  
\usepackage{geometry}
\usepackage{cancel}

\theoremstyle{plain}
\newtheorem{theorem}{Theorem}[section]      
\newtheorem{lemma}[theorem]{Lemma}          
\newtheorem{proposition}[theorem]{Proposition}
\newtheorem{corollary}[theorem]{Corollary}
\newtheorem{claim}[theorem]{Claim}          

\newtheorem*{theorem*}{Theorem}             

\theoremstyle{definition}
\newtheorem{definition}[theorem]{Definition}

\newtheorem{remark}[theorem]{Remark}
\usepackage[hypertexnames=false]{hyperref} 


\numberwithin{equation}{section}

\newcommand{\R}{\mathbb{R}}
\newcommand{\N}{\mathbb{N}}

\begin{document}

\title[Harmonic map-like structure]{Existence and a priori bounds for fully nonlinear PDEs with a harmonic map-like structure}

\author{Gabrielle Nornberg}
\address{Department of Mathematics and Center for Mathematical Modeling (CNRS IRL2807), University of Chile, Chile}
\email{gnornberg@dim.uchile.cl}

\author{Ricardo Ziegele}
\address{Department of Mathematics, Università degli studi di Torino, Italy}
\email{ricardoalfonso.ziegelealiaga@unito.it}


\begin{abstract}
In this paper, we study a new class of fully nonlinear uniformly elliptic equations with a so-called \emph{harmonic map-like} structure, whose model case is given by
\begin{equation*}
\mathcal{M}^{\pm}_{\lambda,\Lambda}(D^2u) \pm b(x) |Du| \pm \beta(u)\langle M(x) Du,Du \rangle \pm c(x) u = f(x)\; \textrm{ in } \Omega,
\end{equation*}
where $\Omega\subset \R^n$ is a bounded $C^{1,1}$ domain, $\mathcal{M}^{\pm}$ are the Pucci extremal operators, $\beta(s) = s^k$ for some $k \in \N $ odd, $b \in L^{q}_{+}(\Omega)$, $c,f \in L^p(\Omega)$, and $n \leq p \leq q$, $q>n$. 

We obtain existence results under a smallness regime on the coefficients, along with some classical results such as the \emph{Aleksandrov--Bakelman--Pucci} estimate and the \emph{comparison principle}, as well as
 \emph{a priori} bounds for the respective Dirichlet problem in the noncoercive case.
We also establish multiplicity results and qualitative behavior, which seem to be new in the case of the Laplacian operator.
\end{abstract}

\maketitle

\section{Introduction}
\label{sec:intro}

Harmonic maps are nonlinear extensions of harmonic functions, arising as critical points of energy functionals restricted to manifolds, which naturally lead to systems of equations involving quadratic gradient terms.
 Such maps arise in physical models such as liquid crystals \cite{komatsu_new_1993, choi_degenerate_1987, chou_constancy_1995} and superfluids \cite{berlyand_homogenization_1999}, and may be applied, for instance, to describe the orientation of cardiac fibers \cite{BARNAFI2025117710}.
On the other hand, their natural motivation comes from geometric analysis (see \cite{helein_harmonic_2008} for a survey on the matter), and from the mathematical point of view, many problems concerning solvability, regularity, and multiplicity of solutions remain open to this day -- see the list of problems proposed by Brezis in \cite{brezis_my_2023}.

In this paper, we study a new class of fully nonlinear uniformly elliptic equations, which we call \emph{harmonic map-like} structure, whose model case is given in the scalar regime by
\begin{equation}
\label{Modelo}
\tag{$M_\beta^{\pm}$}
\mathcal{M}^{\pm}_{\lambda,\Lambda}(D^2u) \pm b(x) |Du| \pm \beta(u)\langle M(x) Du,Du \rangle \pm c(x) u = f(x),
\end{equation}
where $\Omega\subset \R^n$ is a bounded $C^{1,1}$ domain, $\mathcal{M}^{\pm}$ are the Pucci extremal operators, $b \in L^{q}_{+}(\Omega)$, $c,f \in L^p(\Omega)$, with $n \leq p \leq q$. Here, $\beta$ will be a function satisfying 
\begin{equation}
        \label{Hbeta}
        \tag{H$_\beta$}
    \begin{aligned}
    &\beta : \mathbb{R} \to \mathbb{R} \text{ is an odd, locally Lipschitz continuous and nondecreasing} \\
    &\text{function such that } \beta(s)s \geq 0 \text{ for all } s\in \mathbb{R} \text{ and } \lim_{s \to +\infty} \beta(s) = + \infty ,
    \end{aligned}
\end{equation}
where the model $\beta$ satisfying \eqref{Hbeta} to bear in mind is $\beta(s) = s^k$ with $k \in \mathbb{N}$ odd.

Quasilinear equations with a $\beta(u)|Du|^2$ nonlinearity have been studied in several settings. Firstly, the case $\beta \equiv 1$ has been vastly studied since the early 80s (see e.g., 
\cite{arcoya_continuum_2015, boccardo_existence_1982, boccardo_resultats_1984, de_coster_multiplicity_2017, ferone_nonlinear_2000, jeanjean_multiple_2016,
 jeanjean_existence_2013,
 nornberg_priori_2019, sirakov_solvability_2010}, and the references therein), and has become increasingly relevant in recent years, as equations with quadratic growth in the gradient arise naturally in various areas, such as large deviations theory~\cite{boue_variational_1998, possamai_second_2013, robertson_large_2015}, stochastic control~\cite{kobylanski_backward_2000}, and game theory~\cite{cardaliaguet_master_2019}. For results on the case $\beta \not \equiv \text{constant}$, which is the interest of this work, we refer to \cite{porretta_local_2004} for local estimates of \emph{large solutions}, \cite{abdellaoui_remarks_2006} for existence and multiplicity of solutions of
\begin{equation}
\label{abdella}
-\Delta u = \beta(u)|Du|^2 + \lambda f(x), 
\end{equation}
and to \cite{hamid_correlation_2008}, for an extension of these results to the $p$-Laplacian.

Let us state our main results in what follows.
Our first theorem explores the existence of strong $W^{2,p}(\Omega)$ solutions of \eqref{Modelo} in a more general framework where the gradient growth is superlinear, and the function $\beta$ is assumed only to be continuous with polynomial growth. 
\begin{theorem}
\label{teo_existenciasolfuerte}
    Let $\Omega \subset \R^n$ a bounded domain, $\partial \Omega \in C^{1,1}$, $n <p \leq q_1$, $m>1$, $f \in L^p (\Omega)$, $b \in L^{q_1}(\Omega)$, $c \in L^q(\Omega)$, $\psi \in W^{2,p}(\Omega)$ and $\beta \,\colon \R \to \R $ a continuous function such that $|\beta(s)| \leq C_\beta |s|^k$ for some $k \in \mathbb{N}$. Assume that one of the following conditions holds:
    \label{Teoremaclave}
\begin{equation}
\label{lol}
\left\{
\begin{aligned}
&\text{(i)} && q = \infty,  n < p, \\
&\text{(ii)} && n < p \leq q \leq \frac{m}{m-1}p < \infty, \\
&\text{(iii)} && n < p \leq q \text{ and }  q \geq \frac{m}{m-1}p,
\end{aligned}
\right. \\[1em]
\end{equation}
\noindent
and let
\begin{equation}
\label{losR}
\left\{
\begin{aligned}
&r = pm  &&\text{for (i) and (iii)}, \\
&r = \infty &&\text{for (ii)} ; \text{with } p = q, \\
& r = \frac{pq}{q - p} &&\text{for (ii)} ; \text{with } p < q. \\
\end{aligned}
\right.
\end{equation}
\noindent

Then there exists $\varepsilon_1 >0$ such that, if
\begin{equation}
    \label{pequeñezc}
    \|c\|_{L^q(\Omega)} < \frac{\varepsilon_1^{\frac{1}{m+k}}}{C_1 |\Omega|^{\frac{q-p}{pq}}},
\end{equation}
and
\begin{equation}
\label{concluirteo}
  \|\mu\|_\infty   C_\beta C_1^k  (\|f\|_{L^p(\Omega)} + \| \psi \|_{W^{2,p}(\Omega)})^{m+k-1} < \varepsilon_1,
\end{equation}
\noindent
where $C_1$ is the constant associated with the continuous $W^{1,r}(\Omega) \hookrightarrow C(\overline{\Omega})$ embedding, there exists a strong solution $u \in W^{2,p}(\Omega)$ of
\begin{equation}
\label{EQTHEOSWEICH}
    \begin{cases}
    \begin{aligned}
        \mathcal{M}^{\pm}_{\lambda, \Lambda}(D^2u) \pm b(x) |Du| \pm \mu(x) \beta(u)|Du|^m + c(x)u &= f &&\text{ in } \Omega,\\
        \quad \quad \quad \quad \quad \quad  u& = \psi &&\text{ on } \partial \Omega,
        \end{aligned}
    \end{cases}
\end{equation}
\noindent
which satisfies
\begin{equation}
\label{cotainteresante2}
    \|u\|_{W^{2,p}(\Omega)} \leq \hat{C} \left( \| f \|_{L^p(\Omega)} + \| \psi\|_{W^{2,p}(\Omega)}\right),
\end{equation}

\noindent
for some $\hat{C} = \hat{C}(n,\lambda, \Lambda, p,q, q_1, m, \beta, \| b\|_{q_1}, \|c\|_{L^q(\Omega)},\Omega) > 0$.
\end{theorem}
This result recovers a coercivity-type bound on $\|c\|_{L^q(\Omega)}$, in line with the results obtained by Sirakov in \cite{sirakov_solvability_2010}. It further extends the existence results for equations with superlinear gradient terms derived by Koike and Święch in \cite{koike_existence_2009}, even for the case $\beta \equiv 1$, by including in the equation the zero-order term $c(x)u$. In addition, it covers the noncoercive-like case of {\cite[Proposition 3.1]{abdellaoui_remarks_2006}}, whenever $c$ is small.
As an application, we show how this existence result implies an ABP-type estimate and a comparison principle for the new class of fully nonlinear operators, cf.\ Propositions \ref{ABP} and \ref{Comparison} ahead.

Next, we consider the following one-parameter problem:
\begin{equation} 
\label{Plambda}
\begin{cases}
\tag{$\text{P}_\lambda$} 
\begin{aligned}
    -F(x,u,Du,D^2u) &= \lambda c(x)u + \beta(u) \langle M(x)Du, Du\rangle + h(x) &&\text{ in } \Omega\\
    u &= 0 &&\text{ on } \partial\Omega,
\end{aligned}
\end{cases}
\end{equation}
where $\Omega$ is a bounded domain in $\mathbb{R}^n$, $\lambda \in \mathbb{R}$, $n \geq 1$, $c,h \in L^p(\Omega)$, and $\beta(s) = s^k$ with $k \in \N$ odd.

We assume that the matrix $M$ satisfies the non-degeneracy condition
\[
\mu_1 I \leq M(x) \leq \mu_2 I \quad \text{a.e. in } \Omega \tag{M}
\]
for some $\mu_1, \mu_2 >0$, and that $F$ is a uniformly elliptic operator with the following structure:
\begin{equation} 
\label{SC}
\tag{$\text{SC}$} 
\begin{aligned}
    \mathcal{M}^- (X - Y) - b(x) |p - q|  -d(x) \omega \left((r - s)^{+}\right) 
    \leq F(x, r, p, X) - F(x, s, q, Y)    \\  \leq  \mathcal{M}^+(X - Y) + b(x)|p - q| + d(x)\omega((r - s)^+) 
\end{aligned}
\end{equation}
\[
F(\cdot, 0, 0, 0) \equiv 0, \quad b, d, c, h \in L^p(\Omega),\, p > n,\, b, d \geq 0,\, \omega \text{ a modulus of continuity.}
\]

\noindent
Here, $\mathcal{M}^\pm$ are Pucci’s operators with ellipticity constants $0 < \lambda_P \leq \Lambda_P$. We also consider
\[
\mathcal{L}^\pm[u] := \mathcal{M}^\pm(D^2 u) \pm b(x)|Du|,\, b \in L^p_+(\Omega).
\]
For ease of notation, we included in \eqref{SC} the unboundedness condition over the coefficients. On the other hand, we make the convention that $c$ is a bounded function when using the following stronger assumption
\begin{align*}
\label{SC0}
\mathcal{M}^-(X - Y) - b(x)|p - q| &\leq F(x, r, p, X) - F(x, s, q, Y) \\
&\leq \mathcal{M}^+(X - Y) + b(x)|p - q| \quad \text{a.e. } x \in \Omega \tag{SC\textsubscript{0}}
\end{align*}
\[
F(\cdot, 0, 0, 0) \equiv 0, \quad b, h \in L^p(\Omega), \,c \in L^{\infty}(\Omega), \, c,b \geq 0.
\]

Let us consider the problem $(P_0)$ which is \eqref{Plambda} with $\lambda=0$, where $F$ satisfies \eqref{SC}, and we assume that
\begin{equation*}
\label{H0}
(P_0) \text{ has a strong solution } u_0. \tag{H$_0$}
\end{equation*}

Further, setting $F[u] := F(x, u, Du, D^2 u)$, we assume that for each $f \in L^p(\Omega)$,

\[
\text{there exists a unique } L^p\text{-viscosity solution of }
\left\{
\begin{aligned}
    -F[u] &= f(x) \quad \text{in } \Omega \\
    u &= 0 \quad \text{on } \partial \Omega.
\end{aligned}
\right. \tag{H$_1$} \label{H1}
\]

Our main results on the a priori bounds for \eqref{Plambda} read as follows.
The first one provides an a priori lower bound for any solution of \eqref{Plambda}, showing that the negative part is uniformly controlled. The second result establishes uniform upper bounds for positive solutions of problem \eqref{Plambda}.
\begin{theorem}
\label{Propimportante}
    Let $\Omega \in C^{1,1}$ be a bounded domain. Suppose that $\beta(s) = s^k$ with $k \in \N$ odd,   \eqref{SC0} and \eqref{Hbeta} hold,  and let $\Lambda_2 >0$. Then, every $L^p$-viscosity super solution of $($\ref{Plambda}$)$ satisfies $$\|u^-\|_{L^\infty (\Omega)} \leq C \quad \forall \; \lambda \in [0,\Lambda_2],$$
    \noindent
    where $C = C(n,p,\mu_1, \Lambda_2, \|b\|_{L^p (\Omega)}, \|c\|_{L^\infty (\Omega)}, \|h^-\|_{L^p (\Omega)}, \lambda_P, \Lambda_P, \Omega)$.
\end{theorem}

\begin{theorem}
    \label{teoapriori2}
   Let $\Omega \in C^{1,1}$ be a bounded domain. Suppose \eqref{SC0},  \eqref{H0} hold, let $\beta(s) = s^k$ for some $k\in \N$ odd, and $\Lambda_1, \Lambda_2$ with $0 < \Lambda_1 < \Lambda_2$. Then there exists $\delta_0 = \delta_0(u_0,\beta)$ such that, if $\mu_2 < \delta_0,$ then every positive $L^p$-viscosity solution $u$ of (\ref{Plambda}) satisfies $$  \|u\|_{L^{\infty}(\Omega)} \leq C, \; \; \text{for all } \lambda \in [\Lambda_1, \Lambda_2], $$
    where $C$ depends on $n,p,\mu_1, \Omega, \Lambda_1, \Lambda_2, \|b\|_{L^{q}(\Omega)}, \|c\|_{L^{\infty}(\Omega)}, \|h\|_{L^p (\Omega)}, \|u_0\|_{L^{\infty}(\Omega)}, \lambda_P, \Lambda_P, \Omega$, and on the set $G := \{x \in \Omega \; | \; c(x) = 0 \}$.
\end{theorem}

In the case of the Laplacian operator ($\lambda_P=\Lambda_P=1$), we consider the one-parameter problem:
\begin{equation}
\label{Plambdaprima}
    \tag{P$'_\lambda$}
    \begin{cases} 
    \begin{aligned}
    -\Delta u - b(x) |Du| &= \lambda c(x) u + \mu \beta(u) |Du|^2 + h (x) \quad \text{in } \Omega\\
    u &= 0 \hspace{4.5cm}\text{on } \partial \Omega,
    \end{aligned}
    \end{cases}
\end{equation}
in the sense of weak solutions with $u \in H_0^1(\Omega) \cap L^{\infty}(\Omega)$, in the following framework:
\[\begin{cases} 
\label{H}
\tag{H}
\Omega \subset \mathbb{R}^n, \ n \geq 3 \text{ is a bounded domain with } \partial \Omega \in C^{1,1}, c, h \in L^p(\Omega),\\ 
 b \in L^q_+(\Omega) \text{ with } n < p \leq q, \ c \gneqq0 , \,\beta \text{ as in \eqref{Hbeta}},\,\text{and } \mu >0,\\
\end{cases}\]
or more strongly,
\[\begin{cases} 
\label{HS}
\tag{HS}
\Omega \subset \mathbb{R}^n, \ n \geq 3 \text{ is a bounded domain with } \partial \Omega \in C^{1,1}, \\ 
c \in L^{\infty}(\Omega), h \in L^p(\Omega),b \in L_+^{q}(\Omega),  \text{ with } n < p \leq q , \\ c\gneqq0, 
\mu >0, \text{and } \beta(s) = s^k \text{ with $k \in \mathbb{N}$ odd.}
\end{cases}\]
Note that \eqref{HS} provides a model case $\beta(s) = s^k$ which automatically satisfies \eqref{Hbeta}.

Let us also consider the problem $(P_0^\prime)$ (which is \eqref{Plambdaprima} with $\lambda=0$), and a hypothesis on its solvability:
\begin{equation}
    \tag{H$'_0$}
    \label{H0prima}
    \text{$(P_0^\prime)$ has a weak solution }u_0 \in H^{1}_0(\Omega)\cap L^{\infty}(\Omega).
\end{equation}
   
   When $\beta(s) = s^k$, the existence condition in \eqref{H0prima} is satisfied by Theorem \ref{teo_existenciasolfuerte}, as far as   
   \[
      \mu  \, \|h\|_{L^p(\Omega)}^{k+1} < \varepsilon,
   \]
   where $\varepsilon = \frac{1}{(2 \tilde{C} D)^{k+2} C_1^k}$, with $C_1, \tilde{C}$, and $D$ given in Remark \ref{remark31}, see also {\cite[Proposition 3.1]{abdellaoui_remarks_2006}}.

Notice that solutions in $H_0^1 (\Omega)$ of \eqref{P0prima} are not expected to be unique in general, by the multiplicity result in {\cite[Corollary 3.6]{abdellaoui_remarks_2006}} for small $h$ (which amounts to $\lambda f(x)$ in \cite{abdellaoui_remarks_2006}) and $M(x)\equiv I$. 
Observe that this contrasts with the case of decreasing $\beta$ in which comparison always holds, cf.\ \cite[Lemma 2.2]{arcoya_remarks_2014}, and therefore the uniqueness of solutions is expected.

Under our assumption \eqref{H}, though, any weak solution of \eqref{P0prima} belonging to $H_0^1(\Omega)\cap L^\infty(\Omega)$ actually lies in $W^{2,p}(\Omega)$ by our regularity result, cf. Lemma~\ref{lemma:51}. Within this higher regularity class and for small values of $M(x)$, our comparison principle applies (see Remark~\ref{subsuperleq}), and therefore the solution $u_0$ of \eqref{P0prima} is unique.

The following theorem gives a global existence and multiplicity result for problem \eqref{Plambdaprima} with $\lambda >0$.
    \begin{restatable}{theorem}{multi}
    \label{multpositivas}
   Assume \eqref{H0prima}, \eqref{H} and set $\Sigma = \{
          (\lambda,u) \in \mathbb{R} \times C(\overline{\Omega}) \, \colon \, u \text{ solves \eqref{Plambdaprima}}  \}. $
          Then, there exists a continuum of solutions $\mathcal{C} \subset \Sigma$ such that $\mathcal{C} \cap (- \infty, 0] \times C(\overline{\Omega})$ and $\mathcal{C} \cap [0,+\infty) \times C(\overline{\Omega})$ are unbounded.
          
Assume \eqref{HS}, and \eqref{H0prima} with $u_{0}\geq 0$, $cu_{0} \gneqq 0$, then there exists a $\delta^{\star}>0$ such that, if $\mu < \delta^{\star}$, then every nonnegative solution of \eqref{Plambdaprima} with $\lambda>0$ satisfies $u\gg u_{0}$. Moreover, there is $\overline{\lambda}\in(0,+\infty)$ so that
\begin{enumerate}
    \item[(i)] for every $\lambda\in(0,\overline{\lambda})$, problem \eqref{Plambdaprima} has at least two solutions with
    \begin{itemize}
        \item $0\leq u_{0}\ll u_{\lambda,1}\ll u_{\lambda,2}$;
        \item if $\lambda_{1}<\lambda_{2}$, we have $u_{\lambda_{1},1}\ll u_{\lambda_{2},1}$;
        \item $\max u_{\lambda,2}\to+\infty$ and $u_{\lambda,1}\to u_{0}$ in $C^{1}_{0}(\overline{\Omega})$ as $\lambda\to 0$;
    \end{itemize}
    \item[(ii)] problem $(P_{\overline{\lambda}})$ has at least one nonnegative solution $u$;
    \item[(iii)] for every $\lambda>\overline{\lambda}$,  problem $(P_{\lambda})$ has no nonnegative solution.
\end{enumerate}
    \end{restatable}

Theorem  \ref{multpositivas} seems to be new in the Laplacian setting. 
In fact, to the best of our knowledge, no previous works have addressed the question of multiplicity for equation \eqref{Plambdaprima} in what concerns the noncoercive case $\lambda>0$. In particular, our result complements the results for $\lambda=0$ established by Abdellaoui, Dall’Aglio and Peral in \cite{abdellaoui_remarks_2006}, by unveiling the optimal noncoercive parameter for which the multiplicity phenomena occurs under small harmonic-map structure.
On the other hand, this result can be regarded as a natural extension of the case $\beta\equiv 1$ in \cite{de_coster_multiplicity_2017, nornberg_priori_2019}. 

\subsection*{Acknowledgements}
The authors would like to thank Prof.\ Nikola Kamburov for his careful review of the thesis that originated this manuscript and helpful discussions. 

The authors were supported by Centro de Modelamiento Matemático (CMM) BASAL fund FB210005 for center of excellence from ANID-Chile.
R.\ Ziegele was also supported by ANID Fondecyt 1250156.
G.\ Nornberg was also supported by ANID Fondecyt grant 1220776, by Programa Regional MATHAMSUD230019, and by Programa de Cooperación Científica ECOS-ANID 240024.

\section{Auxiliary results}
\label{sec:Auxiliary}

In this section, we explore a change of variables which helps us to deal with the harmonic-map terms  \( \beta(u) |Du|^2 \), in light of the framework developed for quasilinear equations in \cite{abdellaoui_remarks_2006, hamid_correlation_2008, porretta_local_2004}.

\subsection{Variable change: Fully nonlinear setting}
One of the contributions of this work is the introduction of an integral change of variables in the context of our new class of fully nonlinear PDEs, in the spirit of \cite{abdellaoui_remarks_2006} for the Laplacian operator, see also \cite{oliveira_quadratic_2024} for a similar class involving Pucci operators.
This transformation generalizes the exponential change of variables used to deal with pure quadratic gradient terms in the viscosity sense (see {\cite{sirakov_solvability_2010}, also \cite[Lemma 2.20]{saller_nornberg_methods_2018}). 

\begin{proposition}{(Integro-exponential variable change)}
\label{lemma:cv}
    Let $p \geq n$, $\beta$ as in \eqref{Hbeta} and $u \in W_{loc}^{2,p}(\Omega)$ a nonnegative function. For $m>0$ we define: $$v = \psi(u) := \int_{0}^{u} e^{m\int_{0}^{s} \beta(t) dt}ds, \quad w = \Psi (u) := \int_{0}^{u} e^{-m\int_{0}^{s} \beta(t) dt}ds. $$

    Then, a.e. in $\Omega$ we have $Dv = \psi'(\psi^{-1}(v))Du$, $Dw = \Psi'(\Psi^{-1}(w)) Du$, and
    \begin{equation}
        \mathcal{M}_{\lambda, \Lambda}^{\pm}(D^2u) + m \beta(u)\lambda |Du|^2 \leq \frac{\mathcal{M}^{\pm}_{\lambda,\Lambda}(D^2v)}{\psi'(\psi^{-1}(v))} \leq  \mathcal{M}_{\lambda, \Lambda}^{\pm}(D^2u) + m \beta(u)\Lambda |Du|^2, 
        \label{ineq:lema311}
    \end{equation}
        \begin{equation}
        \mathcal{M}_{\lambda, \Lambda}^{\pm}(D^2u) - m \beta(u)\Lambda |Du|^2 \leq \frac{\mathcal{M}^{\pm}_{\lambda,\Lambda}(D^2w)}{\Psi'(\Psi^{-1}(w))} \leq  \mathcal{M}_{\lambda, \Lambda}^{\pm}(D^2u) - m \beta(u)\lambda |Du|^2 .
        \label{ineq:lema312}
    \end{equation}

    Furthermore, the same inequalities hold in the $L^p$-viscosity sense if $u$ is continuous. For example, if $u \in C(\Omega)$ is a nonnegative $L^p$-viscosity solution of
    \begin{equation}
        \mathcal{M}_{\lambda,\Lambda}^{+}(D^2u) + b(x) |Du| +  \mu\beta(u) |Du|^2 + c(x)u \geq f(x) \quad \text{in } \Omega
        \label{eq:lema311}
    \end{equation}
    where $b \in L^{q}_{+}(\Omega)$, $c,f \in L^p(\Omega)$ and $n \leq p \leq q$, then $v = \psi(u)$ with $m = \frac{\mu}{\lambda}$ (resp. $m = \frac{\mu}{\Lambda})$, is an $L^p$-viscosity solution of
    \begin{equation}
        \mathcal{M}^{+}_{\lambda,\Lambda}(D^2 v) + b(x) |Dv| - g(v) f(x) + c(x)(g(v)+1)\psi^{-1}(v) \geq f(x) \quad \text{in } \Omega
        \label{eq;lema312}
    \end{equation}
where $ g(v) = \psi'(\psi^{-1}(v)) - 1 $ is a positive, nondecreasing, and convex function satisfying
$
g(v) \geq c v - 1
$
for some constant \( c > 0 \) and all \( v \geq 0 \),  and
\begin{equation}
\label{gintegral}
    g(v) =m\int_0^v \beta(\psi^{-1}(t)) dt.
\end{equation}
\end{proposition}

\begin{proof}
First, let us establish the basic properties of 
$
g(v) = \psi'(\psi^{-1}(v)) - 1.
$
It is immediate that $g$ is nonnegative, since $\psi'(\psi^{-1}(v)) \geq 1$, with equality attained only at $v = 0$. 
Differentiating, we obtain
\[
g'(v) = \left(\psi'(\psi^{-1}(v)) - 1 \right)' = m\beta(\psi^{-1}(v)).
\]
From this expression it follows that $g'(v) > 0$ for all $v \geq 0$, and further that $g$ is convex, as both $\beta$ and $\psi^{-1}$ are increasing functions.

Next, to prove that $g(v) \geq cv -1$, we claim that there exists a constant $C>0$ such that $$\psi_1'(\psi_1^{-1}(v_1)) \geq C v_1.$$
To prove it, we define $ F(v_1) := \frac{\psi_1'(\psi_1^{-1}(v_1))}{v_1}$. The function \(F\) is continuous on \((0,\infty)\), and moreover
\[
\lim_{v_1 \to 0^+} F(v_1) = +\infty, 
\qquad 
\lim_{v_1 \to +\infty} F(v_1) = +\infty.
\]
Hence, by continuity, \(F\) attains a minimum on \((0,\infty)\). Setting $C := \min_{v_1>0} F(v_1) > 0$, we obtain for every \(v_1>0\) that
\[
\psi_1'(\psi_1^{-1}(v_1)) = F(v_1)\, v_1 \geq C v_1,
\]
\noindent
proving the claim. Thus, since $g(v) = \psi'(\psi^{-1}(v)) -1$, we conclude that $g(v) \geq Cv -1$.
Finally, identity \eqref{gintegral} follows from \(g(0) = 0\) and  
$
g'(v) = m\beta\bigl(\psi^{-1}(v)\bigr) ,
$
by the Fundamental Theorem of Calculus.

To prove \eqref{ineq:lema311} and \eqref{ineq:lema312}, we proceed by direct computations. 
Indeed, since 
$
\psi'(u) = e^{m \int_0^u \beta(t)\, dt},
$
\[
\psi''(u) = m \beta(u) e^{m \int_0^u \beta(t)\, dt} 
           = m \beta(u)\, \psi'(u),
\]
and so,
$
Dv = \psi'(u)\, Du, \;
D^2 v       = m \beta(u)\, \psi'(u)\, Du \otimes Du + \psi'(u)\, D^2 u.
$
Consequently,
\begin{equation*}\begin{aligned}
\mathcal{M}^{\pm}_{\lambda,\Lambda}(D^2v) &= \psi'(u) \mathcal{M}^{\pm}_{\lambda,\Lambda}(m \beta(u) Du \otimes Du + D^2u) \quad \text{(since $\psi'(u) >0$)}.
\end{aligned} \end{equation*}

\noindent
Then, using basic properties of Pucci operators (see for instance {\cite[Lemma 2.1]{sirakov_solvability_2010}}) and the fact that $spec(Du \otimes Du) = \{0, \dots, |Du|^2 \}$, for $u$ nonnegative and $\beta$ satisfying \eqref{Hbeta} we obtain that,
\begin{equation*}
\begin{aligned}
    \mathcal{M}_{\lambda,\Lambda}^{\pm}(D^2u) + \lambda m \beta(u) |Du|^2 \leq \frac{\mathcal{M}_{\lambda,\Lambda}^{\pm}(D^2v)}{\psi'(\psi^{-1}(v))}\leq  \mathcal{M}_
    {\lambda,\Lambda}^{\pm}(D^2u) +\Lambda m \beta(u)|Du|^2.
    \end{aligned} \end{equation*}
Meanwhile, for $u$ nonpositive with $\beta$ satisfying \eqref{Hbeta},
\begin{equation}
    \mathcal{M}_{\lambda,\Lambda}^{\pm}(D^2u) + \Lambda m \beta(u) |Du|^2 \leq \frac{\mathcal{M}_{\lambda,\Lambda}^{\pm}(D^2v)}{\psi'(\psi^{-1}(v))}\leq  \mathcal{M}_
    {\lambda,\Lambda}^{\pm}(D^2u) + \lambda m \beta(u) |Du|^2.
\end{equation}
\noindent
Analogously, $\Psi'(u) = e^{-m \int_0^u \beta(t)dt}$, $\Psi'' (u) = - m\beta(u) \Psi'(u)$, and so
$
    Dw = \Psi'(u) Du, \;
    D^2w  = -m\beta(u)\Psi'(u) Du \otimes Du + \Psi'(u) D^2 u,
$
and $$\mathcal{M}^{\pm}(D^2w) = \Psi'(u) \mathcal{M}^{\pm}(- m\beta(u) Du \otimes Du + D^2u) \quad \text{(since $\Psi'(u) >0$)} ,  $$
\noindent
and, via similar calculations, we obtain that if $u$ is nonnegative and $\beta$ satisfies \eqref{Hbeta}, then
\begin{equation}
\mathcal{M}_{\lambda, \Lambda}^{\pm}(D^2u) - m \beta(u)\Lambda |Du|^2 \leq \frac{\mathcal{M}^{\pm}_{\lambda,\Lambda}(D^2w)}{\Psi'(\Psi^{-1}(w))} \leq  \mathcal{M}_{\lambda, \Lambda}^{\pm}(D^2u) - m \beta(u)\lambda |Du|^2 ,
\end{equation}
while for $u$ nonpositive with $\beta$ satisfying \eqref{Hbeta},
\begin{equation}
\mathcal{M}_{\lambda, \Lambda}^{\pm}(D^2u) - m \beta(u)\lambda |Du|^2 \leq \frac{\mathcal{M}^{\pm}_{\lambda,\Lambda}(D^2w)}{\Psi'(\Psi^{-1}(w))} \leq  \mathcal{M}_{\lambda, \Lambda}^{\pm}(D^2u) - m \beta(u)\Lambda |Du|^2 .
\end{equation}

Now, let $u \in C(\Omega)$ be a positive $L^p$-viscosity solution of \eqref{eq:lema311}. For simplicity, we will consider $m=1$. Also, let $\varphi \in W_{loc}^{2,p}(\Omega)$ such that $v-\varphi$ attains a local maximum in $x_0 \in \Omega$, namely $v \leq \varphi$ and $v(x_0) = \varphi(x_0)$. Let $\varepsilon > 0$ and $\Omega'$ a subdomain such that $x_0 \in \Omega'$ and $\Omega' \subset \subset \Omega $. Set $a = \|u\|_{L^{\infty} \left(\Omega ' \right)}$.\\

\noindent
Define $\phi = \psi^{-1}(\varphi) \in W^{2,p}_{loc}(\Omega)$ (hence $D\phi = \frac{D\varphi}{\psi'(\psi^{-1}(\varphi))}$). Then, by construction,  $u \leq \phi$ and $u-\phi$ attains a maximum in $x_0$. By using \eqref{ineq:lema311} with the pair $\varphi,\phi $, we get
$$ \frac{\mathcal{M}^{+}_{\lambda,\Lambda}(D^2 \varphi)}{\psi'(\psi^{-1}(\varphi))} \geq \mathcal{M}^{+}_{\lambda,\Lambda}(D^2 \phi) + \beta(\phi)|D \phi|^2 \geq \mathcal{M}^+_{\lambda,\Lambda}(D^2 \phi) + \beta(u)|D\phi|^2.$$

\noindent
Exploiting the fact that $u$ is a subsolution of \eqref{eq:lema311} we get

\begin{equation}
    \frac{\mathcal{M}^{+}_{\lambda,\Lambda}(D^2\varphi)}{\psi'(\psi^{-1}(\varphi))} \geq \mathcal{M}^{+}_{\lambda,\Lambda}(D^2 \phi) + \beta(u)|D \phi|^2 \geq f(x) - b(x) |D \phi| - c(x) u - \tilde{\varepsilon} \quad \text{a.e.  } \mathcal{O}
    \label{eq:ch3dem1}
\end{equation}
where $\tilde{\varepsilon} = \frac{\varepsilon}{\psi' (a) +1}$ and $\mathcal{O} \subset \Omega'$ is an open set of positive measure.\smallskip

Now let $\delta \in (0,1)$. Since $(v-\varphi)(x_0) = 0$, $\varphi \in W^{2,p}(\mathcal{O}) \subset C(\mathcal{O})$ (since $p \geq n$) and $\psi$ is a continuous strictly increasing function, with $\psi'$ and $\psi^{-1}$ being also continuous, thus
$$ v \leq \varphi \quad \iff \quad \psi'(\psi^{-1}(v)) \leq \psi' (\psi^{-1}(\varphi)) \in C(\mathcal{O}), $$
and there exists an $\mathcal{O}_\delta \subset \mathcal{O}$ such that, 
$$\psi'(\psi^{-1}(v)) \leq \psi'(\psi^{-1}(\varphi)) \leq \psi'(\psi^{-1}(v)) + \delta \quad \text{in } \mathcal{O}_\delta.$$

\noindent
Then, from \eqref{eq:ch3dem1} we obtain
$$
\begin{aligned}
\mathcal{M}^{+}_{\lambda,\Lambda}(D^2 \varphi) + \underbrace{b(x)|D\phi|\psi'(\psi^{-1}(\varphi))}_{= \;b(x) |D \varphi|} \geq& f(x) \psi'(\psi^{-1}(\varphi)) - c(x) u \psi'(\psi^{-1}(\varphi)) - \tilde{\varepsilon}\psi'(\psi^{-1}(\varphi)),\\
\geq& f(x) \psi'(\psi^{-1}(\varphi)) - c(x) u \left(\psi'(u) + \delta \right) - \tilde{\varepsilon}\left( \psi'(u) + \delta\right).
\end{aligned}
$$
\noindent
By rearranging the terms, we have
$$
\mathcal{M}^{+}_{\lambda,\Lambda}(D^2 \varphi) + b(x) |D\varphi| \geq f^{+} \psi'(u) - f^{-} \left(\psi'(u) + \delta\right) \hspace{5.5cm}
$$
\[
\hspace{5cm} + \left(c^+ u^- + c^{-} u^+ \right)\psi'(u) - \left( c^{+} u^{+} + c^{-} u^{-}\right) \left(\psi'(u) + \delta \right) - \varepsilon \quad \text{a.e. } \mathcal{O}_{\delta},
\]  
i.e., we have shown that $v$ is a subsolution in the $L^p$-viscosity sense of
\begin{equation*}
    \mathcal{M}^{+}_{\lambda,\Lambda}(D^2v) +b(x) |Dv| \geq f(x)\left(g(v) +1 \right) - c(x) u \left(g(v) + 1 \right) - \left(f^- + c^+ u^+ + c^{-} u^{-} \right) \delta \quad \text{a.e. in } \Omega, 
\end{equation*}

\noindent
for every $\delta \in (0,1)$. The desired conclusion follows by standard stability of viscosity solutions, by letting $\delta \rightarrow 0$ and so $\| f_{\delta}\|_{L^{p}(B)} \rightarrow 0 $, for any $B \subset \subset \Omega$,  where $f_\delta := (f^- + c^+ u^+ + c^{-} u^{-})\delta$. The remaining inequalities in the $L^p$-viscosity sense are proven in a similar way.
\end{proof}

\begin{remark}\label{remark:cv}
We can obtain Proposition \ref{lemma:cv} without a sign condition on $u$, but assuming that $\beta$ satisfies the following variant of hypothesis \eqref{Hbeta}, which is given by 

\begin{equation}
    \label{Hbetaprima}
    \tag{H$_\beta'$}
    \begin{aligned}
        &\beta : \mathbb{R} \to \mathbb{R^+} \text{ is locally Lipschitz, and so that } \lim_{x \to +\infty} \beta(x) = + \infty.
    \end{aligned}
\end{equation}

The model $\tilde{\beta}$ satisfying \eqref{Hbetaprima} is given by $\tilde{\beta}(s) = \max\{0,\beta(s)\}$, where $\beta$ satisfies \eqref{Hbeta}. In this case, the inequalities that hold in Proposition \ref{lemma:cv} are those for nonnegative solutions.
\end{remark}

\subsection{Variable change: The Laplacian setting}

First, as a direct consequence of the previous proposition for the Laplacian operator, a variable change which comprises equalities instead of inequalities is reached, with no sign condition on the solutions.
\begin{corollary}
\label{cor:cv}
  If $\lambda = \Lambda = 1$, Proposition~\ref{lemma:cv} yields
$$
    \frac{\Delta v}{\psi'(\psi^{-1}(v))} = \Delta u + m\beta(u)|Du|^2, 
\quad
    \frac{\Delta w}{\Psi'(\Psi^{-1}(w))} = \Delta u - m\beta(u)|Du|^2, 
$$
in the viscosity sense, with no sign condition required on $u$. 
\end{corollary}

Next, we consider the case of weak solutions, in the $H^{1}_0(\Omega)$ sense, used in \cite{abdellaoui_remarks_2006, hamid_correlation_2008, porretta_local_2004}.
We provide a proof in what follows that extends {\cite[Lemma 3]{jeanjean_existence_2013}} to our case, for completeness.

\begin{proposition}
\label{lapla} Let $b \in L_+^q(\Omega)$, and $c, f \in L^p(\Omega)$, with
$
    n < p \leq q.
$    
 Then   $u \in H_0^1(\Omega)\cap L^\infty(\Omega)$ is a weak solution of
\begin{equation}
    \Delta u + b(x)|Du| + \mu \beta(u)|Du|^2 + c(x)u = f(x) 
    \quad \text{in } \Omega,
\label{eqlap}
\end{equation}
if and only if $v = \psi(u) \in H^1_0(\Omega) \cap L^{\infty}(\Omega)$ with $m = \mu$ is a weak solution of
\begin{equation}
\label{eqlap2}
    \Delta v + b(x)|Dv| - g(v) f(x) + c(x)\big((g(v)+1)\psi^{-1}(v)\big) = f(x) 
    \quad \text{in } \Omega,
\end{equation}
where $g(v) = \psi'(\psi^{-1}(v)) - 1$ is a nonnegative even function satisfying 
$g(0)=0$. Moreover, $g$ is decreasing on $(-\infty,0)$, increasing on $(0,+\infty)$, 
and convex in $\mathbb{R}$.
\end{proposition}

\begin{proof}
Let $u \in H_0^1(\Omega) \cap L^{\infty}(\Omega)$ be a weak solution of \eqref{eqlap}. This means that
\begin{equation} \label{eqdebil}\int_\Omega [-Du \cdot D\varphi + b(x) |Du|\varphi + \beta(u) \mu |Du|^2 \varphi + c(x)u\varphi] = \int_\Omega f \varphi \quad \text{for any } \varphi \in C_0^{\infty}(\Omega) \end{equation}
Setting 
\[
\varphi := \phi\, \psi'(\psi^{-1}(v)) = \phi(g(v) + 1),
\]
with $\phi \in C_0^{\infty}(\Omega)$, we have \(\varphi \in H^1_0(\Omega)\), since both \(\psi\) and \(\psi^{-1}\) are \(C^{2}(\mathbb{R})\) functions, and \(v = \psi(u) \in H^1_0(\Omega)\) (since $u\in L^{\infty}(\Omega)$, thus we have that $\psi(u) \in L^{\infty}(\Omega)$ and consequently $Dv = \psi'(u) Du \in L^2(\Omega)$).
Hence, we can use $\varphi$ as a test function in \eqref{eqdebil}. By the definition of $\varphi$, as well as $Du = \frac{Dv}{\psi'(\psi^{-1}(v))}$, and $|Du|\varphi = |Dv| \phi$, we get
\small
\begin{equation}
\label{eqdebil2}
    \begin{aligned}
        \int_\Omega \left[\frac{-Dv \cdot D\big(\phi\, \psi'(\psi^{-1}(v)))}{\psi'(\psi^{-1}(v))} + b(x) |Dv|\phi + \beta(\psi^{-1}(v)) \mu \frac{|Dv|^2}{(\psi'(\psi^{-1}(v)))} \phi + c(x) \psi^{-1}(v)\phi (\psi'(\psi^{-1}(v)) \right]\\ = \int_\Omega f \phi \psi'(\psi^{-1}(v)).
    \end{aligned}
\end{equation}
\normalsize
Now, we observe that
\begin{equation*} \begin{aligned}
D\big(\phi\, \psi'(\psi^{-1}(v))) = 
 D\phi\,\psi'(\psi^{-1}(v)) + \phi \frac{\mu \beta(\psi^{-1}(v)) \psi'(\psi^{-1}(v))}{\psi'(\psi^{-1}(v))}Dv
= D\phi\,\psi'(\psi^{-1}(v))+ \phi \mu\beta(\psi^{-1}(v))Dv.
\end{aligned}\end{equation*}
By using the latter into \eqref{eqdebil2}, yields
\begin{equation}
    \int_\Omega \left[-Dv \cdot D\phi + b(x) |Dv|\phi + c(x) \psi^{-1}(v) (g(v)+1) \phi \right] = \int_\Omega f (g(v)+1) \phi,
\end{equation}
which implies that $v$ is a weak solution of \eqref{eqlap2}.

The other direction is proven similarly, by considering the test function $\varphi = \frac{\phi}{\psi'(u)}$. Concerning the properties of $g$, its evenness comes from the fact that $\psi'$ is an even function (since $\beta$ is odd) and $\psi^{-1}$ is odd. Hence $\psi'(\psi^{-1}(-v)) = \psi'(-\psi^{-1}(v)) = \psi'(\psi^{-1}(v)) $. Moreover, it is trivially nonnegative since $\psi'(\psi^{-1}(v)) \geq 1$, attaining equality only at $v=0$. Finally, computing $g'$ as in the proof of Proposition \ref{lemma:cv}, we see that
$g'(v) = \beta(\psi^{-1}(v)),$
from where it is easy to see that $g'<0 $ in $\mathbb{R}^-$, $g'> 0$ in $\mathbb{R}^+$ and that it is convex, since $\beta$ and $\psi^{-1}$ are nondecreasing.
\end{proof}

\section{Fully nonlinear theory}

\subsection{Proof of Theorem \ref{teo_existenciasolfuerte}}
 We will only show the result for  $\mathcal{M}^+_{\lambda,\Lambda}$, since for $\mathcal{M}^-_{\lambda,\Lambda}$ is analogous. 
    For simplicity, we take $\mu \equiv 1$. Note that, in each of the cases of \eqref{losR}, if $u \in W^{1,r}(\Omega)$, then $\beta(u)|Du|^m \in L^{p}(\Omega)$. Similarly, if $c \in L^q(\Omega)$, then $c(x) u \in L^{q} (\Omega) \subset L^p(\Omega)$. Moreover, the  embedding of $W^{2,p}(\Omega)$ into $W^{1,r}(\Omega)$ is compact. Thanks to {\cite[Proposition 2.4]{koike_existence_2009}}, we can define the mapping $T : W^{1,r}(\Omega) \rightarrow W^{2,p}(\Omega)$ such that for every $v \in W^{1,r}(\Omega)$, it maps it to the unique strong solution $u = Tv$ of 
    \begin{equation*}
        \begin{cases}
        \begin{aligned}
            \mathcal{M}^{+}_{\lambda,\Lambda}(D^2u) + b(x) |Du| 
            &=f(x) - \beta(v(x))|Dv(x)|^m - c(x)v(x) &&\text{ in } \Omega,\\
            u &= \psi &&\text{ on } \partial\Omega.
            \end{aligned}
        \end{cases}
    \end{equation*}
    \noindent
    It follows from {\cite[Proposition 2.4]{koike_existence_2009}} that
    \begin{equation}
    \label{331}
        \begin{aligned}
             \|Tv\|_\infty \leq& \|\psi\|_{L^{\infty}(\partial \Omega)} + C \left(\|f\|_{L^p(\Omega)} + \| \beta(v) |Dv|^m\|_{L^p(\Omega)} + \|cv\|_{L^p(\Omega)}\right)        \\
            \leq& \|\psi\|_{L^{\infty}(\partial \Omega)} + C \left(\|f\|_{L^p(\Omega)} + \beta(\|v\|_{C(\overline{\Omega})})\|Dv\|_{L^r(\Omega)}^m +  |\Omega|^{\frac{q-p}{pq}}\|c\|_{L^q(\Omega)} \|v\|_{C(\overline{\Omega})}\right) \end{aligned}
    \end{equation}
    \noindent
    and
    \begin{equation}
    \label{cotainteresante}
        \|Tv\|_{W^{2,p}(\Omega)} \leq \tilde{C} \left( \| \psi\|_{W^{2,p}(\Omega)} + \|f\|_{L^p(\Omega)} + \beta(\|v\|_{C(\overline{\Omega})})\|Dv\|_{L^r(\Omega)}^m + |\Omega|^{\frac{q-p}{pq}}\|c\|_{L^q(\Omega)} \|v\|_{C(\overline{\Omega})}\right).
    \end{equation}
\noindent
    We will use the Schauder Fixed Point Theorem to prove the existence of a strong solution of \eqref{EQTHEOSWEICH}. To do so, we have to prove that $T$ is continuous in $W^{1,r}(\Omega)$ and that for some $R>0$, $T(\mathcal{B}_{R})$ is a precompact subset of $\mathcal{B}_{R}$, where $\mathcal{B}_{R} := \{ v \in W^{1,r}(\Omega) \; : \; \|v\|_{W^{1,r}(\Omega)} \leq R \}$. \\

    \textit{Continuity of $T$:} Let $v_k \rightarrow v$ in $W^{1,r}(\Omega)$ as $k \rightarrow \infty$. We define, for every $k$, $u_k = Tv_k$. By \eqref{cotainteresante} we get
    $$\begin{aligned}
    \|Tv_k\|_{W^{2,p}(\Omega)} &\leq \tilde{C} \left( \| \psi\|_{W^{2,p}(\Omega)} + \|f\|_{L^p(\Omega)} + \beta(\|v\|_{C(\overline{\Omega})}) \|Dv_k\|_{L^r(\Omega)}^m +|\Omega|^{\frac{q-p}{pq}}\|c\|_{L^q(\Omega)} \|v_k\|_{C(\overline{\Omega})}\right)\\
    &\leq \tilde{C} \left( \| \psi\|_{W^{2,p}(\Omega)} + \|f\|_{L^p(\Omega)} + \beta(\|v_k \|_{C(\overline{\Omega})}) \|v_k\|_{W^{1,r}(\Omega)}^m + |\Omega|^{\frac{q-p}{pq}}\|c\|_{L^q(\Omega)} \|v_k\|_{C(\overline{\Omega})}\right) \\\
    &\leq \tilde{C} \left( \| \psi\|_{W^{2,p}(\Omega)} + \|f\|_{L^p(\Omega)} +  \beta(C_1\|v_k \|_{W^{1,r}(\Omega)}) \|v_k\|_{W^{1,r}(\Omega)}^m + C_1 |\Omega|^{\frac{q-p}{pq}}\|c\|_{L^q(\Omega)} \|v_k\|_{W^{1,r}(\Omega)}\right)\\
    &\leq K \quad \text{for some constant $K > 0$},
    \end{aligned}$$

    \noindent
    by using the fact that $|\beta(s)|\leq C_\beta |s|^k$ and that $W^{1,r}(\Omega)$ is continuously embedded into $C(\overline{\Omega})$ by Morrey's inequality (see, for instance, {\cite[Theorem 7.26(ii)]{gilbarg_elliptic_2001}}) since $r>n$. So, this implies that $Tv_k \rightharpoonup u$ in $W^{2,p}(\Omega)$ for some $u \in W^{2,p}(\Omega)$. Hence, by Rellich-Kondrachov theorem (e.g., {\cite[Theorem 6.3]{adams_sobolev_2003}}), $Tv_k \rightarrow u$ in $W^{1,r}(\Omega)$ since $W^{2,p}(\Omega) \subset\subset W^{1,r}(\Omega)$ for any of the three cases of $r$ in \eqref{losR}. 

Moreover, setting $g_k(x) = f(x) - \beta(v_k(x))|Dv_k(x)|^m - c(x) v_k(x)$ and $g(x) = f(x) - \beta(v(x))|Dv(x)|^m - c(x) v(x)$, we have
    $$
    \begin{aligned}
        \|g_k - g \|_{L^p(\Omega)} &\leq \| \beta(v_k) |Dv_k|^m - \beta(v)|Dv|^m + c\left(v_k - v \right))\|_{L^p(\Omega)} \\
        &\leq \| \beta(v_k) \left(|Dv_k|^m - |Dv|^m \right)  + |Dv|^m \left(\beta(v_k) -  \beta(v)\right) \|_{L^p(\Omega)} + \|c(v - v_k)\|_{L^p(\Omega)}  \\
        &\leq  \| \beta(v_k) \left(|Dv_k|^m - |Dv|^m \right) \|_{L^p(\Omega)} + \||Dv|^m \left(\beta(v_k) -  \beta(v) \right)\|_{L^p(\Omega)} + \|c(v - v_k)\|_{L^p(\Omega)} \\
        &=: I_1 + I_2 + I_3.
    \end{aligned}$$
  Now it remains to prove that $I_1 \rightarrow 0$, $I_2 \rightarrow 0$ and $I_3 \rightarrow 0$. For $I_1$, and for any of the three cases of $r$ in \eqref{losR}, the following holds
\[
\begin{aligned}
I_1 &\leq \| \beta(v_k) \big(|Dv_k| - |Dv|\big) \big(|Dv_k|^{m-1} + |Dv|^{m-1}\big) \|_{L^p(\Omega)} \\
&\leq C \, \|Dv_k - Dv\|_{L^r(\Omega)}
\longrightarrow 0 \quad \text{as } k \to \infty,
\end{aligned}
\]  
for some real constant $C>0$.

Then, for $I_2$ we get 
    $$
    \begin{aligned}
        I_2 &\leq \|Dv\|^m_{L^r(\Omega)} \|\beta(v_k) - \beta(v) \|_{L^r(\Omega)} 
        \longrightarrow 0 \quad \text{as $k \rightarrow \infty$}.
    \end{aligned}
    $$
    \noindent
    Finally, for $I_3$, for any $r$ as in \eqref{losR}, it holds
    $$ \begin{aligned}
        I_3 &\leq |\Omega|^{\frac{q-p}{pq}}\|c\|_{L^q(\Omega)} \|v_k - v\|_{L^r(\Omega)} \\
        &\leq |\Omega|^{\frac{q-p}{pq}}\|c\|_{L^q(\Omega)}\|c\|_{L^q(\Omega)} \|v_k - v\|_{W^{1,r}(\Omega)} 
        \longrightarrow 0 \quad \text{as $k \rightarrow \infty$}.
    \end{aligned}$$

    Therefore $\|g_k-g\|_{L^p(\Omega)} \rightarrow 0$ as $k$ goes to infinity, which implies by the maximum principle ({\cite[Proposition 2.8]{koike_maximum_2007}}) that $Tv_k \rightarrow Tv$ in $C(\overline{\Omega})$ and for uniqueness of the limit, we have $Tv=u$, thus concluding that $T$ is continuous in $W^{1,r}(\Omega)$.\\

    \textit{$T:\mathcal{B}_R \rightarrow \mathcal{B}_R$ for some $R >0$:} Let $\| v \|_{W^{1,r}(\Omega)} \leq R$. By the Sobolev embeddings and {\cite[Proposition 2.4]{koike_existence_2009}}, we have
    $$
    \begin{aligned}
        \|Tv\|_{W^{1,r}(\Omega)} &\leq D \|Tv \|_{W^{2,p}(\Omega)}\\ &\leq D \tilde{C} \left( \| \psi\|_{W^{2,p}(\Omega)} + \|f\|_{L^p(\Omega)} + \beta(\|v\|_{C(\overline{\Omega})})\|Dv\|_{L^r(\Omega)}^m + |\Omega|^{\frac{q-p}{pq}}\|c\|_{L^q(\Omega)} \|v\|_{C(\overline{\Omega})}\right) \\
        &\leq D \tilde{C} \left( \|f\|_{L^p(\Omega)} + \| \psi \|_{W^{2,p}(\Omega)} + \beta(\|v\|_{C(\overline{\Omega})}) R^m + C_1|\Omega|^{\frac{q-p}{pq}}\|c\|_{L^q(\Omega)} \|v\|_{W^{1,r}(\Omega)}\right)\\
        &\leq D \tilde{C} \left( \|f\|_{L^p(\Omega)} + \| \psi \|_{W^{2,p}(\Omega)} +  \beta(C_1\|v\|_{W^{1,r}(\Omega)} ) R^m + C_1|\Omega|^{\frac{q-p}{pq}}\|c\|_{L^q(\Omega)} R\right)\\
        &\leq D \tilde{C} \left( \|f\|_{L^p(\Omega)} + \| \psi \|_{W^{2,p}(\Omega)} +  \beta(C_1 R) R^m + C_1|\Omega|^{\frac{q-p}{pq}}\|c\|_{L^q(\Omega)} R\right) \\
        &\leq D \tilde{C} \left( \|f\|_{L^p(\Omega)} + \| \psi \|_{W^{2,p}(\Omega)} + C_{\beta}C_1^k R^{m+k}+ C_1 |\Omega|^{\frac{q-p}{pq}}\|c\|_{L^q(\Omega)} R\right).
    \end{aligned}
    $$
Set $R = \alpha(\|f\|_{L^p(\Omega)} + \| \psi \|_{W^{2,p}(\Omega)} )$, where $\alpha = 3D\tilde{C}$. Then, $\|Tv\|_{W^{1,r}(\Omega)}$
is controlled by
\small
\begin{equation*}\begin{aligned} 
\leq &\frac{\alpha \left( \|f\|_{L^p(\Omega)} + \|\psi\|_{W^{2,p}(\Omega)} \right) }{3} \left(1 + \alpha^{m+k}  C_\beta C_1^{k}\left( \|f\|_{L^p(\Omega)} + \|\psi\|_{W^{2,p}(\Omega)} \right)^{m+k-1} + \alpha C_1 |\Omega|^{\frac{q-p}{pq}}\|c\|_{L^q(\Omega)}\right).
\end{aligned}\end{equation*}
\normalsize
\noindent
Hence, if \eqref{pequeñezc} and \eqref{concluirteo} are satisfied for $\varepsilon_1\leq (3D \tilde{C})^{-(m +k)}$, then $T(\mathcal{B}_R) \subseteq \mathcal{B}_R$. \\

\noindent
On the other hand, \eqref{cotainteresante} ensures that $T(\mathcal{B}_R)$ is precompact in $W^{1,r}(\Omega)$. Therefore, by the Schauder Fixed Point Theorem, we conclude that $T: \mathcal{B}_R \rightarrow \mathcal{B}_R$ has a fixed point which is, as established before, a strong solution of \eqref{EQTHEOSWEICH}.
Finally, it is easy to see that estimates \eqref{331} and \eqref{cotainteresante}, together with \eqref{concluirteo}, yield the estimate \eqref{cotainteresante2} for some constant $\hat{C}$.
\qed

\begin{remark}
\label{remark31}
One can take $\varepsilon_1 = (3
\tilde{C}D)^{-(m+k)}> 0,$ where $\tilde{C}$ is the constant given in {\cite[Proposition 2.4]{koike_existence_2009}} and $D$ denotes the constant associated with the Sobolev embedding 
$
W^{2,p}(\Omega) \hookrightarrow W^{1,r}(\Omega).
$

\end{remark}

\subsection{ABP and comparison principle}
\label{sec:abp}

Here we present our ABP-type estimate for\begin{equation}
\label{MS}
\mathcal{M}^{\pm}_{\lambda,\Lambda}(D^2u) \pm b(x) |Du| \pm \mu(x) \beta(u) |Du|^m \pm c(x) u = f(x) \; \; \text{in } \Omega,
\end{equation} 
in light of the techniques used in {\cite[Theorem 3.3]{koike_existence_2009}}.

\begin{proposition}[ABP] 
\label{ABP}
Let $n < p \leq q_1$, and let one of \eqref{lol} hold. Let $\varepsilon_1$, and $\beta$ be from Theorem \ref{Teoremaclave} for $\Omega \subset \mathbb{R}^n$, and set $r$ as in \eqref{losR}. If $u\in C(\overline{\Omega})$ is an $L^p$-viscosity subsolution (resp., supersolution) of
    $$\mathcal{M}^+_{\lambda,\Lambda} (D^2u) + b(x) |Du| + \mu(x)\beta(u)|Du|^m + c(x)u\geq -f(x) \quad \text{in } \Omega^+ $$
    $$(\text{resp., } \mathcal{M}^{-}_{\lambda,\Lambda}(D^2u) - b(x) |Du| - \mu(x)\beta(u) |Du|^m + c(x) u\leq  f(x) \quad \text{in } \Omega^-$$
    \noindent
    and if $c \in L^p(\Omega)$ satisfy the smallness condition \eqref{pequeñezc}, and $b \in L^{q_1}_{+}(\Omega)$, $f \in L^p_{+}(\Omega)$ satisfy

\begin{equation}
   \|\mu\|_{\infty} \beta(\|u\|_{\infty}) (2 \|f\|_p)^{m-1} < \varepsilon_1
\end{equation}
    then there exists a constant $C = C(n,\lambda, \Lambda, p,q, q_1, m, \beta, \| b\|_{q_1}, \|c\|_{L^q(\Omega)}, \|u\|_{\infty})$ such that 
    \begin{equation}
        \max_{\overline{\Omega}} u \leq \max_{\partial \Omega} u + C \|f\|_{L^p( \Omega^{+})}
    \end{equation}
    \begin{equation}
        (\text{resp., } \max_{\overline{\Omega}}(-u) \leq \max_{\partial \Omega} (-u) + C \|f\|_{L^p (\Omega^{-})} ).
    \end{equation}
\end{proposition}

\begin{proof}
    The idea is to apply Theorem \ref{Teoremaclave}. To do so, it suffices to notice that, if $\delta > 0$ is such that $ \|\mu\|_{\infty} \beta (\|u\|_{\infty}) \; 2^{m-1}( \|f\|_p+\delta)^{m-1} < \varepsilon_1$, we find a strong solution $v_\delta \in W^{2,p}(\Omega)$ of
    \begin{equation}
    \label{Eq_vdelta}
        \begin{cases}
        \begin{aligned}
            \mathcal{M}^{-}_{\lambda,\Lambda}(D^2 v_\delta) - b(x) |Dv_\delta| - 2^{m-1} \mu(x)\chi_{\Omega^+}(x)\beta(u)|D v_\delta|^m +c(x) v_\delta &= f(x) \chi_{\Omega^+}(x) + \delta &&\text{in } \Omega, \\
            v_\delta &= 0 &&\text{on } \partial \Omega.
            \end{aligned}
        \end{cases}
    \end{equation}
    \noindent
    Since $\|f \chi_{\Omega^+}+\delta\|_p \leq \|f\|_{L^p(\Omega^{+})} + |\Omega|^{\frac{1}{p}} \delta \leq\|f\|_{L^p (\Omega^+)} + (\text{diam}(\Omega))^{\frac{n}{p}}\delta$, estimate (\ref{cotainteresante}) together with the embedding $W^{2,p}(\Omega) \hookrightarrow C(\overline{\Omega})$ gives
    \begin{equation}
    \label{cotaOP}
        \|v_{\delta}\|_{\infty} \leq \hat{C}(\|f\|_{L^p (\Omega^+)} + (\text{diam}(\Omega))^{\frac{n}{p}}\delta)
    \end{equation}
    for some $\hat{C} = \hat{C}(n, \lambda, \Lambda, p,q, q_1,m,k,\|b\|_{q_1}, \|c\|_{L^q(\Omega)}, \|u\|_{\infty}) > 0$. Setting $w:= u + v_\delta$, we add \eqref{Eq_vdelta} and the equation satisfied by $u$ restricted to $\Omega^+$ to get, formally,
    $$\begin{aligned}
    &\mathcal{M}^+_{\lambda,\Lambda}(D^2 w) + \tilde b(x)|Dw| + 2^{m-1} \|\mu\|_\infty A |Dw|^m + c(x) w\\
        & \ge\mathcal{M}^{+}_{\lambda,\Lambda}(D^2u) + \mathcal{M}^{-}_{\lambda, \Lambda}(D^2 v_\delta)+ {b(x) (|Du| - |Dv_\delta|)} +
        \mu(x)\beta(u)|Du|^m - \mu(x)2^{m-1}\beta(u)|Dv_\delta|^m \\ &+c(x)(u+v_\delta)\geq \delta  \quad \text{in } \Omega^+,
    \end{aligned}$$

  where $A = m\beta(\|u\|_{L^{\infty}(\Omega)})$ and $\tilde b =b+2^{m+1}A \|Dv_\delta\|_\infty^{m-1}$.  
    
    Now, even though $u$ is not a strong solution, we can make the argument rigorous by applying test functions to the equation satisfied by $u$ in order to get that $w$ is an $L^p$-viscosity solution of
    \begin{equation}
        \mathcal{M}^+_{\lambda,\Lambda}(D^2 w) + \tilde b(x)|Dw| + 2^{m-1} \|\mu\|_\infty A |Dw|^m + c(x) w \geq \delta \quad \text{in } \Omega^+.
    \end{equation}
    Therefore, by the definition of viscosity solution, we have
    $$\max_{\Omega^+} w = \max_{\partial\Omega^+} w,$$
    which, together with estimation (\ref{cotaOP}) gives
    
    \begin{equation*}
    \begin{aligned}
        \max_{\overline{\Omega}} u = \max_{\Omega^+} u \leq \max_{\partial\Omega^+} w &\leq \max_{\partial\Omega^+} u + \max_{\partial\Omega^+} v_{\delta}\leq \max_{\partial \Omega} u + \|v_\delta\|_{L^\infty (\Omega)}\\
        &\leq \max_{\partial \Omega} u + \hat{C}\left( \|f\|_{L^p (\Omega^+)} + (\text{diam}(\Omega))^{\frac{n}{p}}\delta \right),
    \end{aligned}
    \end{equation*}
    thus concluding the proof by letting $\delta \rightarrow 0$.
    \end{proof}
More generally, at the end of this section, let us consider operators $F:  \Omega \times \mathbb{R} \times \mathbb{R}^n \times \mathbb{S}_n \to \mathbb{R}$ that satisfy the following so-called \emph{harmonic map-like structure}:
\small
\begin{equation} 
\label{SCbeta}
\begin{cases}
\tag{$\text{SC}_\beta$} 
\begin{aligned}
    \mathcal{M}^-_{\lambda,\Lambda} (X - Y) - b(x) |p - q| - \mu |q|^2 \left|\beta(r) - \beta(s)\right| - \mu  \beta(r)(|p|+|q|) |p-q| - c(x) \omega \left((r - s)^+\right) \\
    \leq F(x, r, p, X) - F(x, s, q, Y) \\
    \leq \mathcal{M}^+_{\lambda,\Lambda}(X - Y) + b(x)|p - q| +   \mu |q|^2 \left|\beta(r) - \beta(s)\right|+  \mu \beta(r)(|p|+|q|) |p-q| + c(x)\omega((r - s)^+), 
\end{aligned}
\end{cases}
\end{equation}
\normalsize
where $\beta:\R \to \R$ is an odd, continuous, locally Lipschitz function that is nondecreasing and such that $\beta(s)s \geq 0$ for all $s \in \R$, and all the coefficients are unbounded.

    \begin{proposition}[Comparison Principle]
\label{Comparison}
Assume $F$ satisfies (\ref{SC}), $f \in L^p(\Omega)$, $\beta$ as in (\ref{Hbeta}) or (\ref{Hbetaprima}) with $|\beta(s)| \leq C_\beta |s|^k $ for some $k \in \mathbb{N}$ odd, $M \in L^{\infty}(\Omega)$  with $\mu_1 I \leq M(x) \leq \mu_2 I \text{ a.e. in } \Omega, \; 0 < \mu_1 \leq \mu_2  $, and $\Omega$ a bounded domain. If $u\in C(\overline \Omega)$ is an $L^p$-viscosity supersolution of 
\begin{equation}
\label{xdd}
\begin{cases}
\begin{aligned}
    -F[u] &= \beta(u) \langle M(x) Du,Du\rangle + f(x) \quad &&\text{in } \Omega \\
    \quad \; \; \,u  \; &= 0 \hspace{1cm} &&\text{on } \partial \Omega ,
    \end{aligned}
\end{cases}
\end{equation}
then, for any $\alpha \in C^1(\overline{\Omega}) \cap W^{2,p}_{loc}(\Omega)$ strong subsolution of \eqref{xdd}, there exists a $\delta_0 = \delta_0(\|\alpha\|_{C^1(\overline\Omega)},\beta)$ such that if $
    \mu_2 < \delta_0
$, then we have $\alpha \leq u \; \text{in }\Omega$.

Similarly, if $u\in C(\overline \Omega)$ is an $L^p$-viscosity subsolution of \eqref{xdd}, then for any $\gamma \in C^1(\overline{\Omega}) \cap W^{2,p}_{\text{loc}}(\Omega)$ strong supersolution of \eqref{xdd} there exists $\delta_1 = \delta_1(\|\gamma\|_{C^1(\overline \Omega)},\beta)$ such that, if the same smallness condition holds, we have $u \leq \gamma \; \text{in } \Omega.$
\end{proposition}

\begin{proof}
Set $v := u - \alpha$ in $\Omega$. Reasoning by contradiction, we assume $\min_{\overline{\Omega}} v = v(x_0) < 0$. As $v \geq 0$ on $\partial \Omega$, we get that $x_0 \in \Omega$. Set $\tilde{\Omega} := \{ x \in \Omega  \; | \; v(x) < 0 \}$, which is an open nonempty set since $x_0 \in \tilde{\Omega}$. Let $\varphi \in W^{2,p}_{\text{loc}}(\tilde{\Omega})$ and $\tilde{x} \in \tilde{\Omega}$ be such that $v - \varphi$ has a minimum at $\tilde{x}$. This is equivalent to saying $u - (\alpha + \varphi)$ has a minimum at $\tilde{x}$, and by the fact that $\alpha + \varphi \in W^{2,p}_{\text{loc}}(\tilde{\Omega})$ together with the definition of $u$ being an $L^p$-viscosity supersolution of (\ref{xdd}), we know that for every $\varepsilon >0$, there exists a certain radius $r>0$ such that, for a.e. $x \in B_r(\tilde{x}) \cap \tilde{\Omega}$,
$$-F(x,u,D(\alpha + \varphi), D^2(\alpha + \varphi)) - \beta(u) \langle M(x)D(\alpha +\varphi), D(\alpha + \varphi) \rangle - f(x) \geq -\varepsilon$$
\noindent
and $-F(x,\alpha,D\alpha,D^2\alpha) - \beta(\alpha) \langle M(x) D\alpha, D\alpha \rangle - f(x) \leq 0$ from the definition of $\alpha$ as a strong subsolution of (\ref{xdd}). Now, by subtracting both and using (\ref{SC}), we get in $\tilde\Omega$,
\begin{equation*}
    \begin{aligned}
        \varepsilon &\geq F(x,u,D(\alpha + \varphi), D^2(\alpha + \varphi)) - F(x,\alpha,D\alpha,D^2\alpha)  + \beta(u)\langle M(x) D(\alpha+\varphi),D(\alpha+\varphi)\rangle \\ &- \beta(\alpha) \langle M(x)D\alpha,D\alpha\rangle
        -c(x) \omega((u - \alpha)^+)
        \\
        &\geq \mathcal{M}^{-}_{\lambda,\Lambda} (D^2 \varphi) - b(x) |D \varphi| + \left(\beta(u) - \beta(\alpha) \right)\langle M(x)D\alpha, D\alpha \rangle + \beta(u) \left \{ 2 \langle M(x)D\alpha, D\varphi \rangle + \langle M(x)D\varphi, D\varphi \rangle\right \} \\
        &\geq \mathcal{M}^{-}_{\lambda,\Lambda} (D^2 \varphi) - b(x) |D \varphi|  + \left(\beta(u) - \beta(\alpha) \right)\langle M(x)D\alpha, D\alpha \rangle  - 2 \delta \mu_2 |D\alpha| |D \varphi| + \mu_1 \beta(u) |D \varphi|^2 \\
        &\geq \mathcal{M}^{-}_{\lambda,\Lambda} (D^2 \varphi) - b(x) |D \varphi| -C_\alpha\langle M(x)D\alpha, D\alpha \rangle |u-\alpha| - 2 \delta \mu_2 |D\alpha| |D \varphi| \\&+ \mu_1 (\beta(v+\alpha) - \beta(v)) |D \varphi|^2 + \mu_1 \beta(v) |D\varphi|^2 \\
        &\geq \mathcal{M}^{-}_{\lambda,\Lambda} (D^2 \varphi) - b(x) |D \varphi| + \mu_2 |D\alpha|^2 C_\alpha v - 2\delta \mu_2 |D\alpha| |D \varphi| - \mu_1 (\delta+C_\alpha |\alpha|) |D \varphi|^2 \\
    \end{aligned}
\end{equation*}
\normalsize
where $\delta = |\beta(2\|\alpha\|_{L^{\infty}(\Omega)}) |$ and $C_\alpha$ is a constant associated with the locally Lipschitz condition of $\beta$ in the set $B_{\alpha} :=\{ s \in \mathbb{R} \; | \; |s| \leq \|\alpha\|_{L^{\infty}(\Omega)}\} $. Then, defining $\tilde{b} = b + 2 \delta \mu_2 |D \alpha| \in L^p_+(\tilde{\Omega})$, and $\tilde{c}(x) = \mu_2 |D\alpha|^2 C_\alpha \in L^p_+ (\tilde{\Omega})$, we get that $v$ is an $L^p$-viscosity supersolution of
\begin{equation*}
    \begin{cases}
    \begin{aligned}
        \mathcal{M}^{-}_{\lambda,\Lambda} (D^2 v) - \tilde{b}(x) |Dv|- \tilde\mu |D v|^2 + \tilde{c}(x)  v &\leq 0\qquad &&\text{in } \tilde{\Omega} \\
        v &\geq 0 &&\text{on } \partial \tilde{\Omega} \subset \partial \Omega \cup \{ v = 0 \}
       \end{aligned}
    \end{cases}
\end{equation*}
where $\tilde\mu = \mu_1 (\delta+C_\alpha |\alpha|)$.
Thus, assuming that $\tilde{c}$ satisfies the smallness condition required cf. \eqref{pequeñezc} which in turn is ensured by the smallness assumption on $\mu_2$, using ABP (Proposition \ref{ABP}) with $f= 0$ gives us $v \geq 0$ in $\tilde{\Omega}$, contradicting the very definition of $\tilde{\Omega}$, hence concluding the proof.
\end{proof}

Finally, we state the following proposition, which shows that every strong solution of 
\(F = f\) is also an \(L^p\)-viscosity solution. As a consequence, Theorem 
\ref{teo_existenciasolfuerte} provides an initial existence result for 
\(L^p\)-viscosity solutions of the extremal problem \eqref{Modelo}.

    \begin{proposition}
    \label{prop:stronglp}
        Assume $F$ satisfies \eqref{SCbeta}, $n < p \leq q$, and $f \in L^p(\Omega)$. Then, if $u \in W_{loc}^{2,p}(\Omega)$ is a strong subsolution (resp. supersolution) of $F(x,u,Du,D^2u) = f(x)$ in $\Omega$, then it is an $L^p$-viscosity sub(super) solution of this equation.
    \end{proposition}

\begin{proof}
    We will only prove it for subsolutions, since the proof for supersolutions is analogous.
    Reasoning by contradiction, i.e., that $u$ is not an $L^p$-viscosity subsolution of $F = f $ in $\Omega$, there exists $\theta, r_0 >0$ and $x \in \Omega$, $\phi \in W^{2,p}(B_{r_0}(x))$ such that $0 = (u-\phi)(x) \geq (u-\phi)(y)$ for every $y \in B_{r_0}(x) \subset\subset \Omega$ and
    $$F(y,u(y),D\phi(y),D^2\phi(y)) - f(y) \leq - 2 \theta \quad \text{a.e. in } B_{r_0}(x).$$
    Moreover, since $u$ is a strong subsolution, we have that 
    $$F(y,u,Du,D^2u) \geq f(y) \quad \text{a.e. in }B_{r_0}(x),$$
    hence, combining the two, we get
    $$F(y,u(y),D\phi(y),D^2\phi(y)) - F(y,u(y),Du(y),D^2u(y)) \leq - 2 \theta. $$
    Now, using \eqref{SCbeta} and setting $v=u-\phi$,
    \begin{equation*}
    \begin{aligned}
       - 2 \theta \geq \mathcal{M}^-_{\lambda,\Lambda}(-D^2v(y)) - b(x) |Dv(y)| - \mu |Du(y)|^2 |\beta(u) - \beta(u)| - \mu \beta(u(y)) (|Du(y)| + |D\phi(y)|) |Dv(y)| 
    \end{aligned}
    \end{equation*}
    \normalsize
    i.e.,
    \begin{equation*}
        \mathcal{M}^+_{\lambda,\Lambda}(D^2v(y)) + b(x) |Dv(y)| + \mu \beta(u(y)) (|Du(y)| + |D\phi(y)|) |Dv(y)| \geq  2 \theta \quad \text{a.e. in } B_{r_0}(x).
    \end{equation*}

Now, setting $\gamma(y) = |\beta(u(y))| (|Du(y)|+|D\phi(y)|)$ and $v_\varepsilon(y) = v(y) - \varepsilon|y-x|^2$ for $\varepsilon >0$ as in \cite[Theorem 3.1]{koike_weak_2009}, we have that $v_\varepsilon$ achieves the strict maximum over $\overline{B_{r_0}}(x)$ at $x$ (since $v(y) \leq 0$) and that, for $0 < \varepsilon \leq \theta/(2n\Lambda)$ (where $\Lambda$ is the ellipticity constant of $\mathcal{M}_{\lambda,\Lambda}^{\pm}$) we have
$$\mathcal{M}^+_{\lambda,\Lambda}(D^2 v_\varepsilon) +\gamma(y) |Dv_\varepsilon| \geq - 2r \varepsilon \gamma(y) + \theta \quad \text{a.e. in } B_r(x), \; \text{for all } r \leq r_0.$$
By $u,\phi \in W^{2,p}(B_r) \subset W^{2,p}(B_{r_0}) $, we have that $\gamma \in L^p_+(B_r)$ with $n < p \leq q$. Therefore, we can apply ABP (Proposition \ref{ABP} with $\mu(x) \equiv  0$) to get
$$0 = \sup_{B_r(x)}v_{\varepsilon} \leq \sup_{\partial B_r(x)}v_\varepsilon + C \varepsilon r^{2-n/p} \|\gamma\|_{L^p(B_r(x))},$$
and since $v(y) \leq 0$ in $B_r(x)$, we finally get that $\sup_{\partial B_r(x)}v_\varepsilon \leq - \varepsilon r^2$, thus
$$0 = \sup_{B_r(x)}v_{\varepsilon} \leq -\varepsilon r^2+ C \varepsilon r^{2-n/p} \|\gamma\|_{L^p(B_r(x))}$$
which, by taking $r >0$ sufficiently small, leads to the desired contradiction.
\end{proof}

\subsection{Proof of Theorem \ref{Propimportante}}

Firstly, notice that if $u$ is solution of $($\ref{Plambda}$)$, then $-u$ and $0$ are $L^p$-viscosity subsolutions of 
    \begin{equation}
    \label{eqtilde}
    \begin{aligned}
        -\tilde{F}(x,U,DU,D^2U) &\leq  \lambda c(x) U +\mu_2 \tilde{\beta}(U) |DU|^2  + h^{-}(x), 
    \end{aligned}
    \end{equation}
     \noindent
     where $\tilde{F}(x,r,p,X) = - F(x,-r,-p,-X)$ and $\tilde{\beta}(x) = \max(0, \beta(x))$ is a function that satisfies \eqref{Hbetaprima}, see Remark \ref{remark:cv}. Indeed, on the one hand, if $U = 0$, then by \eqref{SC} we have that $$-\tilde{F}(x,0,0,0) = 0 \leq h^-(x) $$
    \noindent
    which is always true. On the other hand, if $U= -u$, then $$\tilde{F}(x,-u,-Du,-D^2 u) = -F(x,u,Du,D^2u) = \lambda c(x) u + \beta(u) \langle M(x) Du, Du\rangle + h^+(x) - h^-(x),  $$
    \noindent
    and rewriting this in terms of $U$ we obtain
    \begin{equation*}
    \begin{aligned}
        \tilde{F}(x,U,DU,D^2 U) 
        &\geq - \lambda c(x) U - \beta(U) \langle M(x) DU, DU \rangle - h^-(x)  \quad \text{(by the oddness of $\beta$) }\\
        &\geq - \lambda c(x) U - \tilde{\beta}(U)\mu_2 |DU|^2   - h^{-}(x), 
            \end{aligned}
    \end{equation*} 
    therefore, 
       $$-\tilde{F}(x,U,DU,D^2U) \leq \lambda c(x) U + \tilde{\beta}(U) \mu_2 |DU|^2 + h^-(x)$$
    \noindent
    i.e., $U=-u$ is subsolution of \eqref{eqtilde}. Next, using that $F$ satisfies \eqref{SC0}, $-u$ and $0$ are also $L^p$-viscosity subsolutions of
    \begin{equation}
    \label{EqOdd}
        \begin{cases}
        \begin{aligned}
            \mathcal{M}^+(D^2U) + b(x) |DU| + \mu_2\tilde{\beta}(U) |DU|^2 &\geq -\lambda c(x)U - h^-(x) \quad &&\text{in } \Omega\\
            \hspace{5cm}U &\leq 0 &&\text{on } \partial \Omega.
            \end{aligned}
        \end{cases}
  \end{equation}
     Then $U := u^- = \max \{0,-u \}$ is a subsolution of both problems \eqref{EqOdd}, \eqref{eqtilde}. Furthermore, $U\geq 0$ in $\Omega$ and $U = 0$ on $\partial \Omega $ since  $u|_{\partial \Omega}=0$. Next, we apply the change of variable $$w = \Psi(U) :=\int_{0}^{U} e^{-m \int_{0}^t \tilde{\beta}(s) ds} dt = \int_{0}^U e^{-\frac{m}{k+1} t^{k+1}}dt, \quad \text{with } m = \frac{\mu_2}{\lambda_P}$$
    \noindent
    and from Proposition \ref{lemma:cv}, we know that $w$ is an $L^p$-viscosity solution of 

    \begin{equation}
    \label{lolol1}
        \begin{cases}
            \begin{aligned}
                -\mathcal{L}^+_1[w] &\leq h^-(x) + \lambda c(x) (1-G(w) ) \Psi^{-1}(w) \quad &&\text{in } \Omega\\
                w &= 0 &&\text{on } \partial \Omega,
            \end{aligned}
        \end{cases}
    \end{equation}
    where $\mathcal{L}^+_{1} [w] = \mathcal{L}^+ [w] - h^-(x) G(w)$ and $G(w) = 1 -\Psi' (\Psi^{-1}(w)) = m \int_0^w \beta(\Psi^{-1}(t))dt $ is a positive nondecreasing function. Notice that $\mathcal{L}^+_1 $ is a coercive operator, since $G(w)$ is bounded by a positive constant. 
    
    Also, concerning $w = \Psi(v)$ we have
    \begin{equation*}
        \begin{aligned}
            0 \leq w &= \int_{0}^v e^{-\frac{m}{k+1} t^{k+1}}dt \\
            &\leq  \int_{0}^{\infty}e^{-\frac{m}{k+1} t^{k+1}}dt  \\
            &= \frac{1}{k+1} \left(\frac{k+1}{m}\right)^{\frac{1}{k+1}}\int_{0}^{\infty}e^{-z} z^{-\frac{k}{k+1}} dz  \\
            &= \Gamma\left( \frac{1}{k+1}\right) \frac{1}{k+1} \left(\frac{k+1}{m}\right)^{\frac{1}{k+1}} =: C_\beta,
        \end{aligned}
    \end{equation*}
    \noindent
    where in the third inequality we used the change of variable $z = \frac{m}{k+1}t^{k+1},$ and the definition of the Gamma function $\Gamma(s) = \int_0^{\infty} t^{s-1} e^{-t} dt$. Thus we have that \begin{equation} \label{cotaW1} 0 \leq w \leq C_\beta \quad \text{a.e. in } \Omega. \end{equation}
    \noindent
    Now, we set $w_1 = \int_0^{u_1^-} e^{-\frac{m}{k+1} t^{k+1}}dt$  where $u_1$ is a fixed $L^p$-viscosity supersolution of (\ref{Plambda}). Then, due to (\ref{cotaW1}), we have that $w_1 \in [0,C_\beta]$ and that $w_1$ is an $L^p$-viscosity solution of (\ref{lolol1}). Now, we define $\overline{w} = \sup_{v \in \mathcal{A}}v$, where
    $$\mathcal{A} := \left\{w : w \text{ is a $L^p$-viscosity solution of \eqref{lolol1}, } 0 \leq w \leq C_\beta \right\}.$$
    \noindent
    Clearly, $\mathcal{A} \not = \emptyset$ since $w_1 \in \mathcal{A}$ and $w_1 \leq \overline{w} \leq C_\beta$. Moreover, as a supremum of $L^p$-viscosity solutions (locally bounded), $\overline{w}$ is an $L^p$-viscosity solution of the first inequality of (\ref{lolol1}) (as a variant of Perron's method arguments in \cite{koike_perrons_2005}, see also \cite{nornberg_regularity_2021} for equations with superlinear gradient growth), and by construction, $\overline{w} = 0$ on $\partial \Omega$.\\

    \noindent
    As stated before $\mathcal{L}_1^{+}$ is a coercive operator, and the function
    $$f(x) := f_{\lambda}(x,\overline{w}(x)) = h^- (x) + \lambda c(x)\left(1 - G(\overline{w})\right)\Psi^{-1}(\overline{w}) \quad \in L^p_{+}(\Omega),$$
    \noindent
satisfies
    $$\|f^+\|_{L^p(\Omega)} \leq \|h^-\|_{L^p (\Omega)} + \Lambda_{2} \|c\|_{L^p (\Omega)}C_0,$$
    as long as we have  $(1-G(\overline{w})) \Psi^{-1} (\overline{w}) = \Psi'(\Psi^{-1}(\overline{w})) \Psi^{-1}(\overline{w}) \leq C_0$ for $\overline{w} \in [0, C_\beta)$. To verify the latter, we note that such a claim is equivalent to saying that $\Psi'(z) z \leq C_0$ for $z \in [0,+ \infty)$. Now, note that $$\Theta(z) := \Psi'(z) z = e^{-\frac{m}{k+1} z^{k+1}} z \gneqq 0 \; \; \text{in $[0,+\infty)$}$$
    is a continuous non-negative function in $[0, + \infty)$ with $\Theta(0) = 0$, and $\lim_{z \rightarrow +\infty} \Theta(z) = 0$, thus $\Theta(z)$ is bounded. Hence, we get the desired result $(1-G(\overline{w}))\Psi^{-1}(\overline{w}) \leq C_0$ a.e. in $\Omega$, concluding the claim.\\

    \noindent
    Therefore, by the proof of the boundary Lipschitz bound in \cite[Theorem 2.3]{sirakov_boundary_2018}},
    \begin{equation*} \overline{w}(x) \leq C \|f^+\|_{L^p(\Omega)} \text{dist}(x,\partial \Omega) \rightarrow 0 \quad \text{as } \; \; x \rightarrow \partial \Omega. \end{equation*}
    Thus $\overline{w} \not \equiv C_\beta$. Note that it may be equal to $C_\beta$ at some interior points.\\

    Now, arguing by contradiction, if there were a sequence of supersolutions $u_k$ of (\ref{Plambda}) in $\Omega$ with unbounded negative parts, then there would exist a subsequence such that
$$ u_k^{-}(x_k) = \|u_k^-\|_{L^\infty (\Omega)} \xrightarrow[k \to \infty]{} + \infty, \; \; x_k \in \overline{\Omega}, \; \; x_k \xrightarrow[k \to \infty]{} x_0 \in \overline{\Omega}$$
for some $x_0 \in \overline{\Omega}$ with $x_k \in \Omega$ for large $k$, since $u_k \geq 0$ on $\partial \Omega$. Then we define
$$0 \leq w_k (x_k) = \int_0^{u^-_k (x_k)} e^{-\frac{m}{k+1}t^{k+1}}dt \leq C_\beta.$$
and since $u_k(x_k) \xrightarrow[k\rightarrow \infty]{} \infty $, up to a subsequence and the change of variable $z = \frac{m}{k+1} t^{k+1}$ we obtain \begin{equation} \label{limite1}w_k(x_k) \xrightarrow[k\rightarrow \infty]{}C_\beta.\end{equation}

Moreover, by definition of $\overline{w}$ we have
$w_k \leq \overline{w} \leq C_\beta \text{ a.e. in } \Omega$,
hence, by (\ref{limite1}) we get that, for every $\varepsilon >0$, there exists some $k_0 \in \N$ such that 
$$C_\beta - \varepsilon \leq w_k(x_k) \leq  \overline{w}(x_k) \leq C_\beta, \quad \text{for all $k \geq k_0$}$$
thus $\lim_{k \rightarrow \infty} \overline{w}(x_k) = C_\beta$
and also
$$ \overline{w}(x_0) \geq \varliminf_{x_k \rightarrow x_0} \overline{w}(x_k) = \lim_{k \rightarrow \infty} \overline{w}(x_k) = C_\beta$$
which implies that $x_0 \in \Omega$, since we know by construction that $\overline{w} = 0$ on $\partial \Omega$, and so $\overline{w}(x_0) = C_\beta$. Now, applying the change of variables $z = 1 - \frac{1}{C_\beta}\overline{w}$ (note that $z \geq 0$, since $\overline{w} \leq C_\beta$), then $z$ is a solution of
\begin{equation*}
    \begin{cases}
        -\mathcal{L}_1^-[z] \geq  - \frac{1}{C_\beta}\lambda  c(x) \left( 1 - G\left(\{1-z\}C_\beta \right)  \right) \Psi^{-1}\left(\{1-z\}C_\beta \right) \quad \text{in } \Omega \\
        z \gneqq 0 \text{ in } \Omega, \quad z(x_0) = 0,
    \end{cases}
\end{equation*}
where $\mathcal{L}_1^- := \mathcal{L}^- - \frac{1}{C_\beta}h^-(x)(1- G(\{1-\cdot\}C_\beta) $ is a coercive operator, since 
\[
\frac{1}{C_\beta} \, h^-(x) \bigl(1 - G(\{1-\cdot\} C_\beta)\bigr) \ge 0,
\]
and $0 \le 1 - G(\cdot) \le 1.$

We now claim that this contradicts the following Vásquez-type strong maximum principle from \cite{sirakov_vazquez_2021}:

\begin{lemma}[{\cite[Theorem 1.1]{sirakov_vazquez_2021}}]
\label{SMPBoyan}
    Set $\mathcal{L}_1^- := \mathcal{M}^{-}(D^2u) - b(x) |Du| - c(x) u$ with $b \in L^q_+(\Omega)$, $c \in L^p_+(\Omega)$, $n < p \leq q$. Let $f \in C([0, +\infty)$ be a continuous function such that $f(0) = 0$ and that 
    $$\limsup_{s \to 0^+} \frac{f(s)}{s (\ln(s))^2} < + \infty .$$
Then, if $u$ is an $L^p$-viscosity supersolution of
    \begin{equation*}
    \begin{cases}
        \mathcal{L}^-_1[u] &\leq f(u) \quad \text{in } \Omega \\ 
        \qquad u &\geq 0 \quad \quad \,\;\text{in } \Omega
    \end{cases}\end{equation*}
    such that $ess \inf _B u = 0$ for some $B \subset \subset \Omega$, then $u \equiv 0$ in $\Omega$.
\end{lemma}

To prove the claim, it suffices to prove that the following holds:
    $$\lim_{s \to 0^+} \frac{\Psi'(\Psi^{-1}(\{1-s\}C_\beta))\Psi^{-1}(\{1-s\}C_\beta)}{s (\ln(s))^2} = 0,$$
    and that
    $$\lim_{s \to 0^+} \Psi'(\Psi^{-1}(\{1-s\}C_\beta))\Psi^{-1}(\{1-s\}C_\beta)=0 ,$$
    once this proves that $f(s) := \Psi'(\Psi^{-1}(\{1-s\}C_\beta))\Psi^{-1}(\{1-s\}C_\beta)$ satisfies the hypotheses of Lemma \ref{SMPBoyan}.

    Indeed, since $C_\beta = \lim_{t \to + \infty} \Psi(t)$, we have that 
    $$C_\beta - \Psi(t) = \int_t^{\infty} e^{-m\int_0^s\beta(z) dz}ds=\int_t^{\infty} e^{-m\frac{s^{k+1}}{k+1}}ds.$$
    Now, set $t_s = \Psi^{-1}(\{1-s\}C_\beta)$ (note that $t_s \to + \infty$ when $s \to 0)$. Evaluating in $t=t_s$ we obtain
    \begin{equation}
    \label{asintotico}
        \begin{aligned}
            C_\beta - \Psi(t_s) &= C_\beta s
            =\int_{t_s}^{\infty} e^{-m\frac{s^{k+1}}{k+1}}ds.
        \end{aligned}
    \end{equation}
    Now, by asymptotic estimation for integrals (see, for instance, {\cite[Section 6.3 ``Integration by parts'']{bender_asymptotic_1999}}), we get, after integrating by parts, that
\[
\int_{t_s}^{\infty} \left(-\frac{1}{m z^k }\right) \left(-m z^k e^{-m\tfrac{z^{k+1}}{k+1}}dz\right)=\frac{e^{-m\tfrac{t_s^{k+1}}{k+1}}}{mt_s^k}\,(1 + \mathcal{O}(t_s^{-(k+1)})) 
   \sim \frac{e^{\tfrac{-mt_s^{k+1}}{k+1}}}{mt_s^k},
\]
as $s \to 0^+$. This can be interpreted as the mass of the integral concentrating around $t_s$, since the integrand decays very rapidly. Therefore, combining this with \eqref{asintotico} we obtain
$$\Psi'(t_s) = e^{\frac{-mt_s^{k+1}}{k+1}} \sim C_\beta \,t_s^k \,s.$$
Now, multiplying by $t_s$ on both sides, we obtain
\begin{equation}
\label{oooo}
    \Psi'(t_s) t_s \sim C_\beta  \,t_s^{k+1} \, s,
\end{equation}
and taking logarithm,
$$-m\frac{t_s^{k+1}}{k+1} \sim \underbrace{k \ln(t_s)}_{o (t_s^{k+1})}- \underbrace{\ln (C_\beta)}_{\mathcal{O}(1)} - \ln(s), \quad \text{since } \lim_{t \to \infty} \frac{\ln (t)}{t^{k+1}} = 0, $$
which implies that $t_s^{k+1} \sim  \frac{(k+1)}{m} \ln(s)$. Next, using this together with \eqref{oooo} yields,
$$\frac{\Psi'(t_s) t_s}{s (\ln(s))^2} \sim \frac{C_\beta  \,t_s^{k+1} \, \cancel{s}}{\cancel{s} (\ln(s))^2} \sim \frac{(k+1)}{m} \frac{1}{\ln(s)}, \quad s \ll1. $$
Hence, for sufficiently small $s$ we have that
$$\frac{\Psi'(\Psi^{-1}(\{1-s\}C_\beta))\Psi^{-1}(\{1-s\}C_\beta)}{s (\ln(s))^2} \leq  M \frac{(k+1)}{m} \frac{1}{|\ln(s)|}, $$
for some $M>0$. Therefore, 
$$\lim_{s \to 0^+} \ \frac{\Psi'(\Psi^{-1}(\{1-s\}C_\beta))\Psi^{-1}(\{1-s\}C_\beta)}{s (\ln(s))^2} \leq \lim_{s \to 0^+}  M\frac{(k+1)}{m} \frac{1}{|\ln(s)|} = 0, $$
thus concluding the first part of the claim. Now, for the second limit, it suffices to see that
$$
    \lim_{s \to 0^+} \Psi'(\Psi^{-1}(\{1-s\}C_\beta))\Psi^{-1}(\{1-s\}C_\beta)
    = \lim_{z \to +\infty} e^{-m \frac{z^{k+1}}{k+1}} z
    = 0,
$$
hence, proving the second part of the claim and yielding the desired contradiction. \qed

\begin{remark}
    This claim implies that $f(s) = \Psi'(\Psi^{-1}(\{1-s\}C_\beta))\Psi^{-1}(\{1-s\} C_\beta)$ is not integrable at $0$, which justifies the use of a Vásquez-type SMP.
\end{remark}
\subsection{Proof of Theorem \ref{teoapriori2}}
We begin the proof by fixing $\Lambda_1, \Lambda_2$ such that $0 < \Lambda_1 < \Lambda_2$. \\

Reasoning by contradiction, we suppose that the solutions of (\ref{Plambda}) are not bounded from above in $[\Lambda_1, \Lambda_2]$, i.e., there exists a sequence $u_k$ of nonnegative $L^p$-viscosity solutions of (\ref{Plambda}) such that $$u_k^+(x_k) \xrightarrow{k \rightarrow\infty} + \infty, \; x_k \in \overline{\Omega}, \; x_k \rightarrow x_0 \in \overline{\Omega}$$
where $x_k = \text{arg}max_{x \in \overline{\Omega}} |u_k(x)|$, i.e., $\|u_k\|_{L^{\infty}(\Omega)} = |u_k(x_k)|$. Here, $\|u_k\|_{L^{\infty}(\Omega)} = u_k^+(x_k) + u_k^-(x_k)$, and we know that $u_k^- \equiv 0$.
\begin{claim}
\label{claimapriori}
    Up to changing the blow-up point $x_0$, we can suppose that there is a ball around $x_0$ in which $c$ is not identically zero.
\end{claim}
\begin{proof}
    Let $G$ be a maximal domain where $c \equiv 0$. Evidently, there is no need for such argument if $| \{ c=0 \}| = 0$ or even if $c \gneqq 0$ in a neighbourhood of $x_0$. Suppose, then, that $x_0$ is an interior point of $G$, and so $x_k \in G$ for large enough $k$ (up to considering a half ball in $G$ if $x_0 \in \partial \Omega$, after straightening the boundary via a diffeomorphism). Notice that both $u_k$ and $u_0$ satisfy the same equation \begin{equation}
    \label{auxiliar1}
        -F(x,u,Du,D^2u) = \beta(u) \langle M(x) Du, Du \rangle + h(x) \quad \text{in } G
    \end{equation}
    for each $k \in \mathbb{N}$. Moreover, due to (\ref{SC0}) and \eqref{Hbeta}, we have that $F(x,r,p,X) \leq  F(x,r-a,p,X)$ for any $a \in \mathbb{R}^+$. Therefore, $$v_k := u_k-\sup_{\partial G}u_k, \quad v_0 := u_0 - \inf_{\partial G} u_0$$ where $u_0$ is the solution of $(P_0)$ given by (\ref{H0}), are respectively, supersolution and subsolution of \eqref{auxiliar1}. Also, by construction, we have that $v_0 \geq 0 \geq v_k $ on $\partial G$, so $v_0$ and $v_k$ are respectively $L^p$-viscosity sub and supersolution of $(P_0)$ in $G$, with $v_0$ strong. Thus, we apply Proposition \ref{Comparison} to obtain that $v_0 \geq v_k$ in $G$ and also, noting that
    $$v_k \leq v_0 \iff u_k \leq \sup_{\partial G}u_k + 2\|u_0\|_{L^\infty(\Omega)} \quad \text{in } G$$
    we get
    \begin{equation*}
             \label{aux3}
         \sup_{\partial G } u_k \geq u_k - 2 \|u_0\|_{L^{\infty}(\Omega)} \quad \text{in } G.
    \end{equation*}

    Then, evaluating at $x_k \in G$ we have
     \begin{equation*}
        \sup_{\partial G} u_k \geq u_k(x_k) - 2\|u_0\|_{L^{\infty}(\Omega)}  \xrightarrow{k \rightarrow + \infty} + \infty \,,
         \label{aux2}
    \end{equation*}
    which means that we have a blow-up also at the boundary of G, i.e., that there exists a sequence $y_k \in \partial G$ with $u_k^+(y_k) \rightarrow + \infty$ and $y_k \rightarrow y_0 \in \partial G$, as $k$ goes to infinity. Next, since $G$ is maximal, we have that $\partial G \subset \partial \Omega \cup \partial ( \{c = 0\}$), but $u_k = 0$ on $\partial \Omega$ due to being solution of (\ref{Plambda}). Hence, we can take a ball $B_r(y_0)$ centered at $y_0$ (or once again, up to applying a diffeomorphism, a half ball if $y_0 \in \partial \Omega$) which, by enlarging $r$ if necessary, becomes a neighbourhood that intersects the set $\{c >0\}$. In other words, $B_r(y_0)$ is a set such that $c \gneqq 0$. Thus, up to changing $x_k$ and $x_0$ by $y_k$ and $y_0$, we can suppose that $c \gneqq 0$ in $B_r(x_0)$, or in a half ball if $x_0 \in \partial \Omega$, after the usual straightening of the boundary around $x_0$.
\end{proof}
Now, suppose that we are in the more difficult case, i.e., the one of the half ball. For simplicity, we say (up to rescaling) that $c \gneqq 0$ in $B^+_1 = B_1^+ (x_0)$, with our equation being defined in $B_2^+(x_0) \subset \Omega$.

From now on, we make the convention that, from line to line, the constant $C$ that we will be using may change and it depends on $n,p, \lambda_P, \Lambda_P, \Lambda_1,\Lambda_2, \mu_1, \|b\|_{L^p(\Omega)}, \|h\|_{L^p(\Omega)},$ and $\|c\|_{L^p(\Omega)}$. 

Note that \eqref{SC0} implies that for every $L^p$-viscosity solution $u$ of (\ref{Plambda}), the function $u$ is a nonnegative $L^p$-viscosity solution of 
\begin{equation*}
    \begin{aligned}
        \mathcal{M}^-(D^2 u) - b(x) |Du| &\leq F(x,u,Du,D^2u)
        \leq -\lambda c(x) u -\beta(u) \mu_1 |Du|^2 + h^{-}(x), 
    \end{aligned}
\end{equation*}
by using the fact that $\beta(u) \geq 0$ for $u \geq 0$. Therefore, $u$ is an $L^p$-viscosity supersolution of \begin{equation*}
     \mathcal{M}^-(D^2u) - b(x) |Du| \leq - \lambda c(x) u - \beta(u) \mu_1 |Du| + h^{-}(x).
 \end{equation*}
 
Then, by Proposition \ref{lemma:cv}, the function \begin{equation*}
    v_1 := \psi_1(u) = \int_0^u e^{m_1 \int_0^t \beta(s) ds}dt, \quad \text{where } m_1 = \frac{\mu_1}{\Lambda_P}
\end{equation*}
is a nonnegative $L^p$-viscosity supersolution of
\begin{equation}
\label{eq:1AP}
    \mathcal{L^-}[v_1] \leq f_1(x) \quad \text{in } B_2^+
\end{equation}
where $\mathcal{L}_1^-[v_1] := \mathcal{M}^-(D^2 v_1) - b(x) |Dv_1| - m_1 h^-(x) g(v_1)$, with $g(v_1) = \psi_1'(\psi^{-1}_1(v_1)) -1$, and $$f_1(x) := -\lambda c(x) (g_1(v_1) +1) \psi_1^{-1}(v_1) + h^-(x) \in L^p(\Omega)$$
since $v_1 \in L^\infty(\Omega)$ and $\psi_1$, $\psi_1'$, $\psi_1^{-1}$ are continuous functions.  Notice that, in the set $B^+_2 \cap \{f_1 \geq 0 \}$, we have
$ 0 \leq \lambda c(x) (g_1(v_1) + 1 ) \psi_1^{-1}(v_1) \leq h^-(x)$
and $f_1^+ = |f_1| \leq \big |\lambda c(x) (g_1(v_1) + 1 ) \psi_1^{-1}(v_1)\big|+ h^-(x) \leq 2 h^-(x)$, therefore
\begin{equation}
\label{añau}
    \|f_1^+\|_{L^p(B^+_2)} \leq 2 \|h^-\|_{L^p(B^+_2)} \leq C .
\end{equation}

Next, using {\cite[Theorem 1.1]{sirakov_boundary_2018}} (BQSMP, case $p=q > n$) applied to \eqref{eq:1AP}  
we obtain $c_0, C_0$ positive constants, and $\varepsilon \leq 1$, depending on $n, \lambda, \Lambda, \Lambda_P, p$ and $\|b\|_{L^p (\Omega)}$, such that
\begin{equation*} \begin{aligned} 
I &:= \inf_{B_1^+} \frac{v_1}{x_n} \geq c_0 \left ( \int_{B_{3/2}^+} (f_1^-)^{\varepsilon} \right)^{1/\varepsilon} - C_0 \|f^+_1\|_{L^p(B_2^+)} \\
&= c_0 \left( \int_{B^+_{3/2}} \left\{ \left(\lambda c(x)  (g_1(v_1) + 1) \psi_1^{-1}(v_1) - h^-(x) \right)^+ \right\}^\varepsilon \right)^{1/\varepsilon} - C \\
&\geq c_0  \left( \int_{B^+_{3/2}} \left\{ \left(\lambda c(x)  (c_1 v_1 )  \psi_1^{-1}(v_1) - h^-(x) \right)^+ \right\}^\varepsilon \right)^{1/\varepsilon} - C \quad \text{(since $g_1(s) \geq c_1s-1$})\\
&\geq c_0 \inf_{B_1^+}  \frac{v_1}{x_n} \left( \int_{B_1^+} \left\{ \left( \lambda c(x) c_1 \psi_1^{-1}(v_1) x_n - \frac{h^-(x)}{I}  x_n\right)^+ \right\}^\varepsilon \right)^{1/\varepsilon} - C. \\
\end{aligned} \end{equation*}
Next, using the fact that $v_1(x) \geq x_n I$ for all $x \in B_1^+$ we get that
\begin{equation*}
    \begin{aligned}
        I&\geq c_0 I \left( \int_{B_1^+} \left \{ \left( \lambda c(x) c_1 x_n \psi_1^{-1} (I x_n) -\frac{h^-(x)}{I} \right)^+\right\}^{\varepsilon}\right)^{1/\varepsilon}-C.\\ 
    \end{aligned}
\end{equation*}
Then, from the latter we deduce:
\begin{equation*}
    I  \left \{1 - c_0 \left ( \int_{B_1^+}  \left(\left( \lambda c_1c(x) x_n \psi_1^{-1}(I x_n) - \frac{h^-(x)}{I} \right)^+ \right)^{\varepsilon} \right)^{1/\varepsilon} \right \} \geq -C \end{equation*}
thus,
\begin{equation}
\label{desincomprobable}
    I  \left \{c_0 \left ( \int_{B_1^+} \left(\left( \lambda c_1 x_nc(x)  \psi_1^{-1}(I x_n) - \frac{h^-(x)}{I} \right)^+ \right)^{\varepsilon} \right)^{1/\varepsilon} - 1\right \} \leq \tilde{C}, \end{equation}
which we claim that is possible only if $I \leq C$, with a constant $C$ that does not depend on $v_1$, nor on $\lambda \in [\Lambda_1, \Lambda_2]$. Indeed, if this were not the case, we can obtain a sequence of supersolutions $v_1^k$ of $$\mathcal{L}_1^-[v_1^k] \leq f_1^k(x) \quad \text{in } B_2^+$$
such that $I_k:= \inf_{B_1^+} \frac{v_1^k}{x_n} \rightarrow + \infty$ when $k$ goes to infinity and (\ref{desincomprobable}) holding with $I$ replaced by $I_k$. Hence, up to a subsequence (and renumbering if necessary), we can assume that $I_k \geq k^2$ and $\frac{\tilde{C}}{k^2}\leq 1$ for all $k \geq k_0$. From here, it follows that

$$\int_{B_1^+} \left( \left(\lambda c_1c(x) x_n  \psi_1^{-1} (I_k x_n) - \frac{h^-(x)}{I_k} \right)^+\right)^{\varepsilon} \leq c_0^{- \varepsilon}  \left(1 + \frac{\tilde{C}}{k^2} \right )^{\varepsilon} \leq C,$$
then, using that $\lambda \geq \Lambda_1$ and that $I_k \geq k^2$, we get

\begin{equation}
\label{cincosiete}
    \int_{B_1^+ \cap \{ x_n \geq 1/k\}}  \left ( \left( \Lambda_1 c_1c(x) x_n- \frac{h^-(x)}{k^2\psi_1^{-1}(k)} \right)^+ \right)^{\varepsilon} \leq \frac{C}{k^2\,\psi_1^{-1}(k)}.
\end{equation}
Since $\psi_1^{-1}$ is increasing, passing the limits as $k \rightarrow + \infty$ we get $$\int_{B_1^+} (x_n c_1\Lambda_1 c(x))^{\varepsilon} dx \leq 0 \quad \Rightarrow \quad \int_{B_1^+} (x_n c(x))^{\varepsilon} dx = 0,$$
since $\Lambda_1 > 0$. This contradicts that $c(x) \gneqq 0$ in $B_1^+$. In order to justify the limit in \eqref{cincosiete}, we can use the dominated convergence theorem to conclude this. Note that, for $\varepsilon = 1$ we can use the theorem directly, due to (\ref{añau}). On the other hand, for $\varepsilon \in (0,1)$ we use Young's inequality to estimate
\begin{equation*}
\begin{aligned}
\left(\left(\lambda c_1c(x) - \frac{h^-(x)}{k^2\psi_1^{-1}(k)}\right)^+\right)^{\varepsilon} \leq \frac{\left(\lambda c_1 c(x) - \frac{h^-(x)}{k^2\psi_1^{-1}(k)} \right)^+}{\frac{1}{\varepsilon}} + \frac{1}{\frac{1}{1-\varepsilon}} &= \varepsilon \left(\lambda c_1c(x) - \frac{h^-(x)}{k^2\psi_1^{-1}(k)} \right)^+ + 1 - \varepsilon  \\
&\leq \left(\lambda c_1c(x) - \frac{h^-(x)}{k^2\psi_1^{-1}(k)} \right)^+ + 1,\end{aligned}\end{equation*}
ensuring the desired convergence, and leading to the aforementioned contradiction. So, we have proven the claim $I \leq C$.

Therefore, by {\cite[Theorem 1.2]{sirakov_boundary_2018}} (BWHI) applied to \eqref{eq:1AP}, we get that there exists other positive constants $\varepsilon$, $c_0$, $C_0$ depending on $n, \lambda_p, \Lambda_p, p$ and $\|b\|_{L^p(\Omega)}$ such that

\begin{equation}
\label{BWHI1}
\left(\int_{B^+_{3/2}}v_1^{\varepsilon}\right)^{1/\varepsilon} \leq c_0 \left ( \int_{B_{3/2}^+} \left(\frac{v_1}{x_n} \right)^{\varepsilon} \right)^{1/\varepsilon} \leq \inf_{B_1^+} \frac{v_1}{x_n} + C_0 \|f_1^+\|_{L^p(B_2^+)} \leq C.
\end{equation}

Now, going back to $u$, by \eqref{Plambda} and \eqref{SC0} we get that $u $ is an $L^p$-viscosity subsolution of \begin{equation*} \begin{aligned} \mathcal{M}^+(D^2 u) + b(x)|Du| \geq F(x,u,Du,D^2u)& \geq -\lambda c(x) u -\beta(u) \mu_2 |Du|^2 - h^+(x). 
\end{aligned} \end{equation*}
Next, we know from Proposition \ref{lemma:cv} that $v_2:= \psi_2(u) = \int_0^u e^{m_2 \int_0^t \beta(s)ds}dt, \; \; \text{with } m_2= \frac{\mu_2}{\lambda_P}$ is an $L^p$-viscosity subsolution of 
\begin{equation}
    \mathcal{M}^+(Dv_2) +b(x) |Dv_2| + \lambda c(x) (1 + g_2(v_2)) \psi_2^{-1}(v_2)  \geq -h^+ (x) (1 + g_2 (v_2)) 
\end{equation}
which, by rearranging the terms, is equivalent to 
\begin{equation}
    \mathcal{M}^+(Dv_2) +b(x) |Dv_2| + \nu(x) v_2 \geq - h^+(x)
\end{equation}
where $\nu(x) = \lambda c(x) \frac{\psi_2'(\psi^{-1}_2(v_2))}{v_2} \psi_2^{-1}(v_2) + h^+(x) \frac{\psi_2'(\psi^{-1}_2(v_2))-1}{v_2} $. Hence, $v_2$ is an $L^p$-viscosity solution of
\begin{equation}
    \label{eqv2}
    \begin{cases}
    \begin{aligned}
        
    \mathcal{M}^+(Dv_2) +b(x) |Dv_2| + \nu(x) v_2 &\geq - h^+(x) \quad \text{in } B_2^+,\\
    v_2 &= 0 \quad \quad \quad \;\; \;\text{in } B_2^0.
    
       \end{aligned}
    \end{cases}
\end{equation}

\begin{claim}
\label{claimclave1}
    For any $s >0$, there exists a constant $C_s$ such that \begin{equation}
    \label{desig.original}
        \lambda c(x) \frac{\psi_2'(\psi^{-1}_2(v_2))}{v_2} \psi_2^{-1}(v_2) + h^+(x) \frac{\psi_2'(\psi^{-1}_2(v_2))-1}{v_2} \leq C_s \left(c(x) + h^+(x) \right)   (1 + v_2^s) 
    \end{equation}
\end{claim}

\begin{proof} Define the auxiliary functions
\[
a(v_2) := \frac{\psi_2'(\psi_2^{-1}(v_2))}{v_2} \psi_2^{-1}(v_2), \quad \text{and} \quad b(v_2) := \frac{\psi_2'(\psi_2^{-1}(v_2)) - 1}{v_2}.
\]
We aim to show that both \( a(v_2) \) and \( b(v_2) \) grow at most polynomially in \( v_2 \), so that inequality \eqref{desig.original} holds. Let us study the behavior of $a,b$ as $v_2 \rightarrow 0$ and $v_2 \rightarrow + \infty$:
\medskip

\noindent\textbf{Asymptotic behavior as \( v_2 \to 0 \):} Let \( w := \psi_2^{-1}(v_2) \), so that \( w \to 0 \) as \( v_2 \to 0 \). Then, using L'hôpital rule is easy to see that
    $$  
    \lim _{v_2 \rightarrow 0} a(v_2) =  \lim _{w \rightarrow 0} \frac{\psi'_2(w) w}{\psi_2(w)} 
     = 1.$$
     Similarly,
     $$ 
         \lim_{v_2 \rightarrow 0} b(v_2) = \lim _{w \rightarrow 0}\frac{\psi_2'(w)-1}{\psi_2(w)}  
          = m_2 \beta(0)=0.
    $$
     Hence, both \( a(v_2) \) and \( b(v_2) \) remain bounded as \( v_2 \to 0 \). 
     \medskip

\noindent\textbf{Asymptotic behavior as \( v_2 \to \infty \):} Let us recall that
\[
\psi_2(w) = \int_0^w e^{m_2 \int_0^t \beta(s)\, ds} dt= \int_0^w \frac{1}{m_2\beta(t)} \left( m_2\beta(t) e^{m_2 \int_0^t \beta(s) ds}\right)dt .
\]
Then, using standard asymptotic estimates for integrals of exponential type again, after integrating by parts, we conclude
for large enough $w$ that,
\begin{equation}
\label{estimacionv2}
\psi_2(w) \sim \frac{e^{m_2 \int_0^w \beta(s) ds}}{m_2 \beta(w)} = \frac{e^{m_2 \frac{w^{k+1}}{k+1}}}{m_2 w^k}=\frac{\psi_2'(w)}{m_2 w^k}.
\end{equation}
Therefore,
\begin{equation}
\label{estrellita}
a(v_2) = \frac{\psi_2'(w) w}{\psi_2(w)} \sim m_2 w^{k+1}, \qquad b(v_2)=\frac{\psi_2'(w) - 1}{\psi_2(w)} \sim m_2 w^k.
\end{equation}
Now, taking logarithm in~\eqref{estimacionv2} and recalling that \( v_2 = \psi_2(w) \), we find
\[
\begin{aligned}
\ln v_2 &\sim \frac{m_2 w^{k+1}}{k+1} - \ln(m_2 w^k) 
= \frac{m_2 w^{k+1}}{k+1} + o(w^{k+1}) \;\;\textrm{ for large $w$},
\end{aligned}
\]
which implies
\begin{equation}
\label{estrellita2}
w = \psi_2^{-1}(v_2) \sim \left( \ln v_2 \right)^{\frac{1}{k+1}}.
\end{equation}
Combining \eqref{estrellita} and \eqref{estrellita2} into our estimates gives
\[
a(v_2) \sim m_2 \ln v_2, \qquad b(v_2) \sim m_2 (\ln v_2)^{\frac{k}{k+1}} \leq m_2 \ln v_2.
\]
Whence, there exist constants \( C_1, C_2 > 0 \) such that for \( v_2 \gg 1 \),
\[
a(v_2) \leq C_1 \ln(1 + v_2), \qquad b(v_2) \leq C_2 \ln(1 + v_2).
\]

Finally, since for every \( s > 0 \) there exists a constant \( C_s > 0 \) such that
$
\ln(1 + v_2) \leq C_s (1 + v_2^s),
$
we conclude that \( a(v_2), b(v_2) \leq \mathcal{O}(1 + v_2^s) \) as \( v_2 \to \infty \). To extend this to all \( v_2 \geq 0 \), let \( M > 0 \) be such that the asymptotics above hold for all \( v_2 \geq M \), and define
\[
C_{1,s} := \max\left\{ \sup_{v_2 \in [0, M]} a(v_2), \, C_1 C_s \right\}, \quad
C_{2,s} := \max\left\{ \sup_{v_2 \in [0, M]} b(v_2), \, C_2 C_s \right\}.
\]
Then for all \( v_2 \geq 0 \), we have
\[
a(v_2) \leq C_{1,s} (1 + v_2^s), \qquad b(v_2) \leq C_{2,s} (1 + v_2^s),
\]
which implies inequality~\eqref{desig.original}.
\end{proof}

Next, from the definition of \( v_2 \), we have
\[
v_2 = \int_0^{\psi_1^{-1}(v_1)} \left( \psi_1'(t) \right)^{\frac{m_2}{m_1}} dt.
\]
Since \( \psi_1' \) is increasing and \( \frac{m_2}{m_1} > 1 \), we estimate
\begin{align*}
    v_2 &\leq \psi_1^{-1}(v_1) \left( \psi_1'(\psi_1^{-1}(v_1)) \right)^{\frac{m_2}{m_1}} = \psi_1^{-1}(v_1) \left( 1 + g_1(v_1) \right)^{\frac{m_2}{m_1}} \\
    &= \psi_1^{-1}(v_1) \left( 1 + m_1 \int_0^{v_1} \left( \psi_1^{-1}(t) \right)^k dt \right)^{\frac{m_2}{m_1}} \quad \text{(by \eqref{gintegral})} \\
    &\leq \psi_1^{-1}(v_1) \left( 1 + m_1 v_1 \left( \psi_1^{-1}(v_1) \right)^k \right)^{\frac{m_2}{m_1}} = \left( \left( \psi_1^{-1}(v_1) \right)^{\frac{m_1}{m_2}} + m_1 v_1 \left( \psi_1^{-1}(v_1) \right)^{k + \frac{m_1}{m_2}} \right)^{\frac{m_2}{m_1}}.
\end{align*}

Now, using the same asymptotic estimates as in Claim~\ref{claimclave1}, we have
\[
\psi_1^{-1}(v_1) \sim \left( \ln v_1 \right)^{\frac{1}{k+1}} \quad \text{as } v_1 \to \infty.
\]
Hence, for large \( v_1 \),
\[
v_2 \leq C\left(  (\ln(1 + v_1))^{\frac{1}{k+1} \cdot \frac{m_1}{m_2}} +  m_1 v_1 (\ln(1 + v_1))^{ \frac{k}{k+1} +\frac{m_1}{m_2} \cdot \frac{1}{k+1}} \right)^{\frac{m_2}{m_1}},
\]
and since the second term dominates asymptotically, we conclude
\[
v_2 \leq \mathcal{O}\left( \left( v_1 \ln(1 + v_1) \right)^{\frac{m_2}{m_1}} \right).
\]

As a consequence, for every \( z > 1 \), there exists a constant \( C_z > 0 \) such that
$
v_2^{\frac{m_1}{m_2}} \leq C_z(1 + v_1^z).
$
We use this to bound \( \|\nu\|_{L^{p_1}(B_2^+)} \) with $p_1 = \frac{p+n}{2}$ and $\frac{1}{p_1} = \frac{1}{p}+\frac{1}{p_2}$. First,
\begin{align*}
\|\nu\|_{L^{p_1}(B_2^+)} &\leq \|c + h^+\|_{L^p(\Omega)} \| C_s(1 + v_2^s) \|_{L^{p_2}(B_2^+)} \\
&\leq C_s \|c + h^+\|_{L^p(\Omega)} + C_s \|c + h^+\|_{L^p(\Omega)} \|v_2^s\|_{L^{p_2}(B_2^+)}.
\end{align*}

Using the previous bound on \( v_2 \), we obtain
\[
\|v_2^s\|_{L^{p_2}(B_2^+)} \leq \|(C_z(1 + v_1^z))^{\frac{m_2}{m_1}s}\|_{L^{p_2}(B_2^+)}.
\]

Now, choosing
$
s = \frac{\varepsilon m_1}{m_2 z p_2},
$
for any \( z > 1 \), we get
\begin{align}
\label{refsz}
\left\| (C_z(1 + v_1^z))^{\frac{m_2}{m_1}s} \right\|_{L^{p_2}(B_2^+)} 
= \left( \int_{B_2^+} \left( C_z(1 + v_1^z) \right)^{\varepsilon / z} \right)^{1/p_2} 
\leq C_\varepsilon \left( C + \int_{B_2^+} v_1^\varepsilon \right)^{1/p_2} \leq C.
\end{align}

Therefore,
\[
\|\nu\|_{L^{p_1}(B_2^+)} \leq C_{n,p} (\|c\|_{L^p(\Omega)} + \|h^+\|_{L^p(\Omega)}) + C.
\]

Having obtained this uniform bound on \( \nu \), we can now apply {\cite[Theorem 2.1]{saller_nornberg_methods_2018}} (BLMP, in the case \( p = q > n \)) to equation~\eqref{eqv2} for \( v_2 \), using \( r = \frac{\varepsilon m_1}{2 m_2} \) as in~\eqref{refsz} (with \( z = 2 \)), to deduce
\[
v_2^+ = v_2 \leq C \left\{ \left( \int_{B^+_{3/2}} v_2^r \right)^{1/r} + \|h^+\|_{L^p(\Omega)} \right\} \leq C,
\]
i.e.,\ $u = \psi^{-1}_2(v_2) $ is uniformly bounded in $B_1^+$, for every $L^p$-viscosity solution of $(P_\lambda)$, $\lambda \in [\Lambda_1,\Lambda_2]$. \qed
\section{Multiplicity results}
This section is devoted to the derivation of \emph{a priori} estimates and multiplicity results for the one-parameter quasilinear problem \eqref{Plambdaprima} (as defined in the introduction) driven by the Laplacian operator.

For this problem we establish new \emph{a priori} estimates, thus extending the results obtained in Theorem \ref{teoapriori2}. Subsequently, we conduct a multiplicity analysis of solutions, generalizing the results obtained in \cite{de_coster_multiplicity_2017} for natural growth in the gradient, but adapted to a harmonic map-type structure like ours.

We recall problem $(P_0^\prime)$ given by \begin{equation}\label{P0prima}\tag{$P_0'$}\begin{cases} \begin{aligned}    -\Delta u - b(x) |Du| &=   \mu \beta(u) |Du|^2 + h (x)  &&\text{in } \Omega,\\   \hspace{2.7cm}u &= 0 &&\text{on } \partial \Omega.   \end{aligned}   \end{cases}\end{equation}

\subsection{Preliminaries}
Before going to the proof of Theorem \ref{multpositivas}, we must reckon some preliminary results and definitions needed over the section, in light of \cite{de_coster_priori_2019}. First, let us consider a class of equations which entails \eqref{Plambdaprima}:
\begin{equation}
\label{caratheodory}
\left \{
\begin{aligned}
-\Delta u &= f(x,u,\nabla u), \quad &&\text{in } \Omega, \\
u &= 0, \qquad \quad &&\text{on } \partial\Omega,
\end{aligned}
\right.
\end{equation}
where $f$ is an $L^p$-Carathéodory function with $p>n$, whose solutions are in $W^{2,p}_0(\Omega)$. 

    It is easily verifiable that, if \eqref{H} or \eqref{HS} holds, the function $f(x,u,Du) := h(x) + c(x) u + b(x)|Du| + \mu \beta(u) |Du|^2 $
    is an $L^p$-Carathéodory function.

We now introduce a strong ordering relation on $C(\overline{\Omega})$ that will be needed for some arguments using degree theory. In order to do so, we define $\varphi_1>0$ as the first eigenfunction of
\begin{equation*}
    -\Delta \varphi_1 = \gamma_1 c(x) \varphi_1, \quad \varphi_1 \in H_0^1(\Omega).
\end{equation*}
With this, we can define the ordering relation appearing in Theorem \ref{multpositivas}:

\begin{definition}[{\cite[Definition 1.2]{de_coster_multiplicity_2017}}]
    Let \( u \), \( v \in C(\overline{\Omega}) \). We say that:
\begin{equation*}
 \text{\( u \ll v \) in case there exists \( \varepsilon > 0 \) such that, for all \( x \in \overline{\Omega} \), \( v(x) - u(x) \geq \varepsilon \varphi_{1}(x) \).}
 \end{equation*}
\end{definition}

Here, \( u, v \in C^{1}(\overline{\Omega}) \),  \( u \ll v \) is equivalent to: for all \( x \in \Omega \), \( u(x) < v(x) \) and, for \( x \in \partial\Omega \), either \( u(x) < v(x) \) or \( u(x) = v(x) \) and \( \frac{\partial u}{\partial \nu}(x) > \frac{\partial v}{\partial \nu}(x) \), where $\nu$ is the outward unit normal.
Also, the set $\{w \in C_0^1(\overline\Omega) \, |\, u \ll w \ll v \} $ is open in $C_0^1(\overline\Omega):= \{ w \in C^1(\overline{\Omega}) \, \colon \, w|_{\partial \Omega} = 0 \}$.

Next, we define a suitable notion of sub and super solutions for problem \eqref{caratheodory}.

\begin{definition}[Regular sub and supersolutions {\cite[Definition 2.1]{de_coster_multiplicity_2017}}]
We will call \emph{regular subsolution} (respectively a \emph{regular supersolution}) of \eqref{caratheodory} as a strong subsolution of the equation in \eqref{caratheodory} satisfying the inequality $\le $ (resp.\ $\ge $) on the boundary, that is,   \( \alpha \) (resp. \( \gamma \)) in \( W^{2,p}(\Omega) \) such that
\[
\left\{
\begin{aligned}
-\Delta\alpha(x) &\leq f(x,\alpha(x),\nabla\alpha(x)) \,\text{a.e. } x \in \Omega, \\
\alpha(x) &\leq 0   \;\text{ for all } x \in \partial\Omega.
\end{aligned}
\right.\;
\left( \textrm{resp. }
\left \{
\begin{aligned}
-\Delta\gamma(x) &\geq f(x,\gamma(x),D\gamma(x)) \,\text{a.e. } x \in \Omega, \\
\quad \; \;\gamma(x) &\geq 0 \;  \text{ for all } x \in \partial\Omega.
\end{aligned}
\right.\right)
\]
\end{definition}

\begin{definition}[Sub and supersolutions, {\cite[Definition 2.2]{de_coster_multiplicity_2017}}]
\label{subsuper}
    We define as a \emph{subsolution} \( \alpha \) of \eqref{caratheodory} a function of the form \( \alpha := \max\{\alpha_{i} \mid 1 \leq i \leq k\} \) where \( \alpha_{1},\ldots,\alpha_{k} \) are regular subsolutions of \eqref{caratheodory}. Similarly, a \emph{supersolution}, \( \beta \) of \eqref{caratheodory} is a function of the form \( \beta = \min\{\beta_{j} \mid 1 \leq j \leq l\} \) where \( \beta_{1},\ldots,\beta_{l} \) are regular supersolutions of \eqref{caratheodory}.
    \end{definition}

Notice that problem \eqref{caratheodory} can be studied as a fixed point problem, since on one hand, the operator
\begin{equation} \label{L}\mathcal{L}: W^{2,p}_0(\Omega) \to L^p(\Omega) \,; \,u \mapsto f(\cdot,u(\cdot), Du(\cdot))\end{equation}
is a linear homeomorphism, and on the other hand, since $f$ is an $L^p$-Carathéodory function, the operator
$$\mathcal{N} : C_0^1(\overline{\Omega}) \to L^p(\Omega) \, ;\, u \mapsto f(\cdot, u(\cdot), Du(\cdot))$$
is well-defined, continuous, and maps bounded sets into bounded sets. Moreover, since $p>n$, $W^{2,p}(\Omega) \subset \subset C_0^1(\overline{\Omega})$ and therefore, the operator $\mathcal{M}: C_0^1(\overline{\Omega}) \to C_0^1(\overline{\Omega}) $ defined by
$\mathcal{M}(u) = \mathcal{L}^{-1} \mathcal{N} u$
is completely continuous. Hence, finding a solution of \eqref{caratheodory} is equivalent to finding $u$ so that
$u = \mathcal{M} u.$

As we are interested in the dependence on the $\lambda$ parameter, we will be referring to $\mathcal{M}_\lambda$ as the operator related to \eqref{Plambdaprima}, and to $\mathcal{M}_0$ as the one related to \eqref{P0prima}.

\begin{definition}[Strict sub/super solutions, {\cite[Definition 2.3]{de_coster_multiplicity_2017}}]
A subsolution $\alpha$ of \eqref{caratheodory} is said to be \emph{strict} if every solution $u$ of \eqref{caratheodory} with $\alpha \leq u$ on $\Omega$ is such that $\alpha \ll u$. Similarly, a \emph{strict} supersolution $\gamma$ of \eqref{caratheodory} is a supersolution such that every solution $u$ with $u \leq \gamma$ satisfies $u \ll \gamma$.

\end{definition}

With these definitions, we recall the following theorem from \cite{de_coster_multiplicity_2017}, which will be instrumental in proving the existence and characterizations of solutions of \eqref{caratheodory} through a sub/super solutions approach.

\begin{theorem}[{\cite[Theorem 2.1]{de_coster_multiplicity_2017}}]
\label{thm:21coster}
Let \(\Omega\) be a bounded domain in \(\mathbb{R}^{n}\) with boundary \(\partial\Omega\) of class \(C^{1,1}\) and \(f\) be an \(L^{p}\)-Carathéodory function with \(p > n\). Assume that there exists a subsolution \(\alpha\) and an supersolution \(\gamma\) of \eqref{caratheodory} such that \(\alpha \leq \gamma\). Denote \(\alpha := \max\{\alpha_{i} \mid 1 \leq i \leq k\}\) where \(\alpha_{1},\ldots,\alpha_{k}\) are regular subsolutions of \eqref{caratheodory} and \(\gamma = \min\{\gamma_{j} \mid 1 \leq j \leq l\}\) where \(\gamma_{1},\ldots,\gamma_{l}\) are regular supersolutions of \eqref{caratheodory}. 

If there exists \(K > 0\) and \(h \in L^{p}(\Omega)\) such that for a.e. \(x \in \Omega\), all \(u \in [\min\{\alpha_{i} \mid 1 \leq i \leq k\}, \max\{\gamma_{j} \mid 1 \leq j \leq l\}]\) and all \(\xi \in \mathbb{R}^{n}\),
$
|f(x,u,\xi)| \leq h(x) + K|\xi|^{2}, 
$
then problem \eqref{caratheodory} has at least one solution \(u\) satisfying
$
\alpha \leq u \leq \gamma.
$

Moreover, problem \eqref{caratheodory} has a minimal solution \(u_{\min}\) and a maximal solution \(u_{\max}\) in the sense that, \(u_{\min}\) and \(u_{\max}\) are solutions of \eqref{caratheodory} with \(\alpha \leq u_{\min} \leq u_{\max} \leq \gamma\) and every solution \(u\) of \eqref{caratheodory} with \(\alpha \leq u \leq \gamma\) satisfies \(u_{\min} \leq u \leq u_{\max}\).

If, in addition, \(\alpha\) and \(\gamma\) are strict and satisfy \(\alpha \ll \gamma\), then there exists \(R > 0\) such that
$
\deg(I - \mathcal{M}, \mathcal{S}) = 1,
$
where
$
\mathcal{S} = \{u \in C^{1}_{0}(\overline{\Omega}) \mid \alpha \ll u \ll \gamma,\ \|u\|_{C^{1}} < R\}.
$
\end{theorem}

\begin{remark}[{\cite[Remark 2.2]{de_coster_multiplicity_2017}}]
\label{remark22}
    If $\alpha$ and $\gamma$ are respectively strict sub and supersolutions of \eqref{caratheodory} with $\alpha \leq \gamma$, then $\alpha \ll \gamma$. 
\end{remark}

Moreover, we recall a well-known result that helps when dealing with quadratic gradient terms.

\begin{lemma}[Nagumo's lemma {\cite[Lemma 2.4]{de_coster_multiplicity_2017}}]
\label{Nagumo}
    Let $p> n$, $h\in L^p(\Omega)$, $K>0$, $R>0$. Then there exists $C>0$ such that, for all $u \in W^{2,p}(\Omega)$ satisfying
    \begin{align*}
|\Delta u| &\leq h(x) + K |\nabla u|^2, \quad \text{a.e. in } \Omega, 
\end{align*}
with $u = 0 $ on $\partial \Omega$ and
$
\|u\|_{\infty} \leq R,
$
we have
$
\|u\|_{W^{2,p}} \leq C.
$
\end{lemma}

Finally, we prove the following lemma, which will be instrumental to obtaining the existence of multiple solutions in the proof of Theorem \ref{multpositivas}.

    \begin{lemma}
    \label{strict}
        Assume \eqref{HS} holds. Then, for every $\lambda \geq0$, and every supersolution $\gamma$ of \eqref{Plambdaprima}, there exists a strict subsolution $v_\lambda$ of \eqref{Plambdaprima} such that $v_\lambda \leq \gamma$.
\end{lemma}

\begin{proof}
First, we fix $\gamma$ supersolution of \eqref{Plambdaprima}. We denote $C_1>0$ be given by Theorem \ref{Propimportante} such that, for every supersolution $\gamma$ of
    \begin{equation}
    \begin{aligned}
    \label{41}
         -\Delta u  &=  \lambda c(x)u + b(x)|Du| + \mu \beta(u) |Du|^2 - h^- (x)-1 \quad &&\text{in } \Omega\\
    \hspace{2.7cm}u &= 0 \hspace{3.55cm}&&\text{on } \partial \Omega,
    \end{aligned}
\end{equation}
we have $\gamma > -C_1$. Moreover, note that every supersolution of \eqref{Plambdaprima} is also a supersolution of \eqref{41}. Now, let $k > C_1$ and consider $\alpha_k$ the solution of
    \begin{equation}
    \label{eqalpha}
    \begin{aligned}
         -\Delta v  &=  -\lambda c(x)k  - \mu\beta(k) A^2 - h^- (x)-1 \quad &&\text{in } \Omega\\
    \hspace{2.7cm}v &= 0 \hspace{3.55cm}&&\text{on } \partial \Omega,
    \end{aligned}
\end{equation}
where $A := \max \{\, \|D \gamma_j\|_{C(\overline{\Omega})} \, | \, 1 \leq j \leq l\, \}$.
Since $- \lambda kc(x) - h^-(x) - \mu \beta(k) A^2 - 1   <0$, we have that $\alpha_k \ll 0$ by the strong maximum principle (see, for instance, {\cite[Theorem 3.27]{troianiello_elliptic_1987}}). Moreover, using elliptic estimates and the embedding $W^{2,p}(\Omega) \hookrightarrow C^{1}(\Omega)$, we get that
{
\begingroup
\renewcommand{\theHequation}{DalphaAnchor}
\begin{equation}
\label{eq:Dalpha}
\|D \alpha_k \|_{C(\Omega)} = \mathcal{O}(k+ \beta(k)), \;\text{for large } k.
\end{equation}
\endgroup
}

Now, we claim that $\gamma$  satisfies $\gamma \geq \alpha_k.$ Indeed, since $\gamma = \min \{\gamma_j \, | \, 1 \leq j \leq l \}$ where $\gamma_1, \dots, \gamma_l$ are regular supersolutions of \eqref{Plambdaprima}, it suffices to notice that, setting $w = \gamma_j - \alpha_k$ for some $1 \leq j \leq l$ we have

    \begin{equation*}
    \begin{aligned}
         -\Delta w  &\geq  \lambda c(x) (\gamma_j + k) + \mu \big[\underbrace{\beta(\gamma_j)|D\gamma_j|^2 +  \beta(k)  A^2}_{\geq 0}  \big] + b(x) |D\gamma_j| \geq 0\quad &&\text{in } \Omega
    \end{aligned}
\end{equation*}
with  $w = 0 $ on $ \partial \Omega$,
which, by the maximum principle, gives us $w \geq 0$, i.e., $\gamma_j \geq \alpha_k$, proving the claim.

Now, consider the truncated problem
    \begin{equation}
    \label{42}
    \begin{aligned}
         -\Delta v  &=  \lambda c(x)T_k(v) + \mu \beta(T_k(v))|Dv|^2 + b(x) |Dv|- h^- (x)-1 \quad &&\text{in } \Omega
    \end{aligned}
\end{equation}
with $v = 0 \text{ on } \partial \Omega,$ where
$$T_k(v) = \begin{cases}
    -k, \quad \text{if } v \leq -k,\\
    \; \, \,\,v, \quad \text{if } v>-k.
\end{cases}$$

We claim that $\alpha_k$ and $\gamma$ are sub and supersolutions of \eqref{42}. In fact, for $\gamma$ the verification is immediate: since $\gamma \geq -C_1 > -k$ and $\gamma$ is already an supersolution of \eqref{41}, it follows that $\gamma$ is also an supersolution of \eqref{42}. For $\alpha_k$, we know that it would be a subsolution if
$$-\Delta \alpha_k \leq \lambda c(x) T_k(\alpha_k) + \mu \beta(T_k(\alpha_k)) |D \alpha_k|^2 + b(x) |D\alpha_k| - h^-(x) -1$$
which is equivalent (up to replacing $-\Delta\alpha_k$ with the RHS of \eqref{eqalpha}) to
\begin{equation} \label{estrellalap}-\lambda c(x) k - \beta(k) \mu A^2  \leq \lambda c(x) T_k(\alpha_k) + \mu \beta(T_k(\alpha_k)) |D \alpha_k    |^2 + b(x)|D\alpha_k|.\end{equation}

Now, as $\alpha_k \ll 0$, for large enough $k$, we have $\alpha_k \leq -k$ and therefore $T_k(\alpha_k) = -k$. Using this, it yields
$$-\mu \beta(k)k \leq - \mu \beta(k) |D\alpha_k |^2 + b(x) |D \alpha_k|.$$
From here, using \eqref{eq:Dalpha}, we can conclude that \eqref{estrellalap} holds for large $k$. Hence, $\alpha_k$ is indeed a subsolution.

Therefore, by Theorem \ref{thm:21coster}, this problem \eqref{42} has a minimal solution $v_k$ with $\alpha_k \leq v_k \leq \gamma$.

\begin{claim}
    $v_k$ is a strict subsolution of \eqref{Plambdaprima}.
\end{claim}
To prove the claim, we first show that $v_k$ is a subsolution. 

Since $v_k$ is a supersolution of \eqref{41}, we have $v_k \geq -C_1 > -k$ and that it satisfies

    \begin{equation*}
    \begin{aligned}
         -\Delta v  &=  \lambda c(x)v + \mu \beta(v)|Dv|^2 + b(x) |Dv|- h^- (x)-1 \quad &&\text{in } \Omega\\
    \hspace{2.7cm}v &= 0 \hspace{3.55cm}&&\text{on } \partial \Omega.
    \end{aligned}
\end{equation*}
This implies that $v_k$ is a subsolution of \eqref{Plambdaprima}.
Finally, to verify it is strict, let $u$ be a solution of \eqref{Plambdaprima} with $u \geq v_k$. Then $w = u - v_k$ satisfies
    \begin{equation*}
    \begin{aligned}
         -\Delta w  &=  \lambda c(x)w + \mu (\beta(u)|Du|^2 - \beta(v_k)|Dv_k|^2) + b(x) |Dw|- h^- (x)-1 + h(x) \quad &&\text{in } \Omega\\
    \hspace{2.7cm}w &= 0 \hspace{3.55cm}&&\text{on } \partial \Omega,
    \end{aligned}
\end{equation*}
which, by using that $h(x) = h^+(x) - h^-(x)$, we have \begin{equation*}
\begin{aligned}
\mu (\beta(u)|Du|^2 - \beta(v_k)|Dv_k|^2)&=\mu(\beta(u) - \beta(v_k))|Du|^2 + \mu \beta(u) (|Du|^2 - |Dv_k|^2), \\
&\geq \mu \beta(u) (|Du| + |D v_k|) |Dw|, 
\end{aligned}
\end{equation*} and after defining $\tilde{b}(x) = b(x) + \mu \beta(u) ( |Du| + |Dv_k|)$, becomes
    \begin{equation*}
    \begin{aligned}
         -\Delta w  - \tilde{b}(x)|Dw|&\geq  \lambda c(x)w+h^+ (x)+1\geq 1  \quad &&\text{in } \Omega
    \end{aligned}
\end{equation*}
with $w = 0 \text{ on } \partial \Omega.$
From here, using once again the strong maximum principle {\cite[Theorem 3.27]{troianiello_elliptic_1987}}, we obtain $w \gg 0$, which is equivalent to $u \gg v_k$. In other words, we proved that $v_k$ is a strict subsolution of \eqref{Plambdaprima}, concluding the proof.
\end{proof}

\subsection{Regularity and \emph{a priori} bounds}
We dedicate this section to the development of \emph{a priori} bounds for the problem \eqref{Plambdaprima}. With this goal, we begin this section with a fundamental regularity result.

\begin{lemma}
\label{lemma:51}
    Let \eqref{H} hold and $u \in H^1_0(\Omega) \cap L^\infty(\Omega)$ be a weak solution of \begin{equation}
    \label{ch5:modelo}
        -\Delta u - b(x)|Du| = c(x) u + \mu \beta(u) |Du|^2 + h(x),
    \end{equation} 
    then $u \in W^{2,p}(\Omega)$.
\end{lemma}

\begin{proof}
    Set $v= \psi_1(u) := \int_0^u e^{m \int_0^t \beta(s) ds} dt $ with $m = \frac{1}{\mu}$. By Proposition \ref{lapla}, we know that $v\in H^1_0(\Omega) \cap L^\infty(\Omega)$ is a weak solution of
    \begin{equation*}
        - \Delta v -b(x)|Dv|=  c(x) (1+g(v)) \psi^{-1}(v) + h(x) (1+g(v)).
    \end{equation*}
    Setting $$\tilde{h}(x) = c(x) (1+g(v)) \psi^{-1}(v) + h(x) (1+g(v)),$$
    we claim that $\tilde{h} \in L^p(\Omega)$. Indeed, since $\psi^{-1}(v) = u \in L^{\infty}(\Omega)$, we have $1+g(v) = \psi'(\psi^{-1}(v)) \in L^{\infty}(\Omega) $. Hence, 
    $$\|\tilde{h}\|_{L^p(\Omega)} \leq \|c\|_{L^p(\Omega)} \|\psi'(\psi^{-1}(v))\|_{L^{\infty}(\Omega)} \|\psi^{-1}(v)\|_{L^\infty (\Omega)} + \|\psi'(\psi^{-1}(v))\|_{L^{\infty}(\Omega)}\|h\|_{L^p(\Omega)} < +\infty. $$
    
    Notice that $v \in H_0^1(\Omega)$ and $n < p \leq q$, allow us to use global $C^{1,\alpha}$ regularity estimates of weak solutions (see {\cite[Chap. 3, Theorem 15.1]{ladyzhenskaya_linear_1968}} for $Lv := \Delta v + b(x)|Dv|$) to get 
    $$\|v\|_{C^{1}(\overline\Omega)} \leq C \|\tilde{h}\|_{L^p(\Omega)}.$$
    
    Now, given that $Dv = \psi'(u) Du \in C(\overline{\Omega})$ we conclude that $Du \in C(\overline{\Omega})$ and also  $|Du|^2 \in L^{\infty}(\Omega)$.
    Finally, $u$ becomes a weak solution of $-\Delta u = \tilde H(x) $ in $\Omega$, with $\tilde H(x):= b(x)|Du|+c(x)u+\mu \beta (u) |Du|^2 +h(x) \in L^p (\Omega)$. Then, standard $W^{2,p}$ regularity estimates for weak solutions give that $u\in W^{2,p}(\Omega)$.
\end{proof}

As a direct consequence of this result, we obtain a connection between weak solutions of \eqref{ch5:modelo}, strong solutions, and $L^p$-viscosity solutions, which is summarized in the following corollary.

\begin{corollary}
\label{cor:visco}
    Assuming \eqref{H} holds, any solution $u \in H_0^1(\Omega) \cap L^\infty(\Omega)$ of
    \begin{equation*}
        -\Delta u - b(x)|Du|= c(x) u + \mu \beta(u) |Du|^2 + h(x)
    \end{equation*} 
    is also an $L^p$-viscosity solution.
\end{corollary}

\begin{proof}
    Indeed, from Lemma \ref{lemma:51} we know that $u \in W^{2,p}(\Omega)$, i.e., $u$ is a strong solution. Therefore, by Proposition \ref{prop:stronglp} we get that $u$ is an $L^p$-viscosity solution, yielding the desired result.
\end{proof}
\begin{remark}
\label{subsuperleq}
 Let \eqref{H} hold. By Proposition \ref{Comparison}, Lemma \ref{lemma:51} and Corollary \ref{cor:visco}, we conclude that if $\alpha$ is a subsolution of \eqref{P0prima} and $\gamma$ a supersolution of \eqref{P0prima}, then there exists $\delta^\star >0$ such that if $\mu < \delta^\star$, then $\alpha \leq \gamma$. The same holds under \eqref{HS}. 
\end{remark}

We now present the \emph{a priori} bounds for solutions of \eqref{Plambdaprima} independent of the sign of the solutions. This theorem can be seen as an extension of Theorem \ref{teoapriori2}, without the nonnegative condition of the solutions.
\begin{theorem}
\label{apriorilaplaciano}
    Let $\Omega\in C^{1,1}$ be a bounded domain. Suppose \eqref{HS}, \eqref{H0prima} hold and let $\Lambda_1, \Lambda_2$ with $0 < \Lambda_1 < \Lambda_2$. Then, there exists $\delta_0 = \delta_0(u_0, \beta)$ such that, if $\mu < \delta_0,$ then every $L^p$-viscosity solution $u$ of \eqref{Plambdaprima} satisfies
    $$\|u\|_{L^{\infty}(\Omega)} \leq C, \, \, \text{for all } \lambda \in [\Lambda_1, \Lambda_2], $$
    where $C = C(n,p,\mu,\Omega, \Lambda_1, \Lambda_2, \|b\|_{L^{q}(\Omega)}, \|c\|_{L^{\infty}(\Omega)}, \|h\|_{L^{p}(\Omega)}, \|u_0\|_{L^{\infty}(\Omega)}, \Omega,  \{x \in \Omega \; | \; c(x) = 0 \}).$
\end{theorem}

\begin{proof}
We begin by fixing $\Lambda_1, \Lambda_2$ such that $0 < \Lambda_1 < \Lambda_2$. 
From Theorem \ref{Propimportante}, there exists $C_1 > 0$ such that \begin{equation*}
\label{fijoC1} u^- \leq C_1 \quad \text{for every $u$ supersolution of \eqref{Plambdaprima}, for all $\lambda \in [0,\Lambda_2]$.}\end{equation*}
Reasoning by contradiction, we suppose that the solutions of \eqref{Plambdaprima} are not bounded from above in $[\Lambda_1, \Lambda_2]$. I.e. there exists a sequence $u_k$ of $L^p$-viscosity solutions of \eqref{Plambdaprima} such that $$u_k^+(x_k) \xrightarrow{k \rightarrow\infty} + \infty, \; x_k \in \overline{\Omega}, \; x_k \rightarrow x_0 \in \overline{\Omega}$$
where $x_k = \text{arg}max_{x \in \overline{\Omega}} |u_k(x)|$, i.e., $\|u_k\|_{L^{\infty}(\Omega)} = |u_k(x_k)|$. Here, $\|u_k\|_{L^{\infty}(\Omega)} = u_k^+(x_k) + u_k^-(x_k)$, and we know that $u_k^- (x_k) \in [0,C_1]$.
Now, as in Claim \ref{claimapriori}, we conclude that, up to changing the blow-up point $x_0$, we may suppose that there is a ball around $x_0$ in which $c$ is not identically zero.

Now, suppose that we are in the case of the half ball, and namely $c \gneqq 0$ in $B^+_1 = B_1^+ (x_0)$, with our equation being defined in $B_2^+(x_0) \subset \Omega$. Now, note that $u$ is an $L^p$-viscosity solution of
$$\Delta u + b(x)|Du| = -\lambda c(x) u - \mu \beta(u) |Du|^2 + h^- (x). $$
Then, by Corollary \ref{cor:cv}, the function
$$v_1 = \psi_1(u) = \int_0^u e^{m\int_0^t \beta(s) ds}dt, \quad \text{with } m = \frac{1}{\mu}$$
is an $L^p$-viscosity solution of
$$\mathcal{L}^{-}[v_1] = f_1(x) \quad \text{in } B_2^+$$
where $\mathcal{L}^-[v_1] := \Delta v_1 -b(x) |Dv_1| - h^-(x) g(v_1),$ with $g(v_1) = \psi'(\psi_1^{-1}(v_1)) -1$, and 
$$f_1(x) := -\lambda c(x) (g_1(v_1) + 1) \psi_1^{-1}(v_1) + h^{-}(x) \in L^p(\Omega)$$
since $v_1 \in L^{\infty}(\Omega)$ and $\psi_1, \psi^{-1}_1 \in C^{2}(\mathbb{R})$. 
\begin{claim}
    There exists a constant $C >0$ such that $\psi_1'(\psi_1^{-1}(v_1)) \geq C v_1.$
\end{claim}
\begin{proof}
We distinguish two cases, accordingly to the sign of $v_1$.

As a first case, suppose \(v_1 \in [\psi_1(-C_1),0]\). Since \(\psi_1(-C_1)<0\), we have
$
\psi_1'(\psi_1^{-1}(v_1)) \geq 0 \geq C v_1,
$
so the inequality holds trivially for any $C>0$.
Otherwise, we refer to the proof of the properties of $g$ in Proposition \ref{lemma:cv}, which, combined with the first case, yields the claim.
\end{proof}
Now, setting $\mathcal{L}_1^-[v_1] := \Delta v_1 - b(x) |Dv_1| - C h^-(x) v_1 $, we have that $v_1$ is an $L^p$-viscosity solution of
$$\mathcal{L}_1^-[v_1] = f_1(x) \quad \text{in } B_2^+,$$
and from here, the remainder of the proof proceeds exactly as in the proof of Theorem \ref{teoapriori2}, only changing that, in this case, $v_1$ can be sign-changing. Since $v_1$ is bounded from below and we just proved that the inequality $g(v_1) \geq Cv_1 -1$ holds, we get the asymptotics proved in the proof of Theorem \ref{teoapriori2}.
\end{proof}
We now present a non-existence result that will be useful in what follows. This result can also be seen as an immediate extension of {\cite[Lemma 3.6]{de_coster_multiplicity_2017}} to our equation structure.

\begin{lemma}
\label{36}
    Assuming \eqref{HS}, for every $\Lambda_2 >0$, there exists $A_1 >0$, independent of $\lambda \in [0,\Lambda_2]$, such that the problem
    \begin{equation}
    \label{33}-\Delta u = \lambda c(x) u+ b(x)|Du| + \mu \beta(u)|Du|^2 + h(x) + a c(x), \quad u \in H_0^1(\Omega) \cap L^{\infty}(\Omega)\end{equation}
    has no solution for $a \geq A_1$.
\end{lemma}

\begin{proof}
We will divide the proof in three cases: $u$ being a sign-changing solution, i.e., $\Omega^+ := \{ x \in \Omega \, | \, u(x) >0\}$ and $\Omega ^- := \{ x \in \Omega \, | \, u(x) \leq 0\}$ both nonempty, $u$ being nonnegative and $u$ being nonpositive.

\textit{\textbf{Case 1:} $u$ is sign-changing.}

Let $u^+:= \max \{u,0 \} \in H_0^1(\Omega) \cap L^{\infty}(\Omega)$ be a test function in \eqref{33}. This yields
\begin{equation}
\label{eqlema}\int_\Omega |Du^+|^2 = \int_\Omega \big [ \lambda c(x) u + b(x) |Du| + \mu \beta(u) |Du|^2 + h(x) + a c(x) \big ] u^+.\end{equation}
Now, using Poincaré inequality and the fact that, by Theorem \ref{apriorilaplaciano}, there exists $K >0$ independent of $\lambda$ such that $\| u \|_{L^{\infty}(\Omega)} \leq K$, we obtain
\begin{equation*}
    \begin{aligned}
   C_p K |\Omega|  \geq C_p \|u^+\|^2_{L^2(\Omega)}\geq \int_{\Omega} |Du^+| ^2,
    \end{aligned}
\end{equation*}
Defining 
$
    C := C_p K\,|\Omega|,
    $
and noting that 
$
    u\,u^+ = (u^+)^2, 
    \;
    Du \cdot Du^+ = |Du^+|^2,
$
it follows that every term on the right-hand side of \eqref{eqlema}, except for the integral 
\(\int_\Omega h\,u^+\), is nonnegative. Therefore, we obtain
\begin{equation*}
    \begin{aligned}
        C &\geq a \int_\Omega c u^+ + \int_\Omega h u^+ + \mu \int_\Omega \beta(u) |Du^+|^2 u^+ + \int_{\Omega} b(x) |Du^+| u^+ + \lambda \int_{\Omega} c (u^+)^2,\\
        &\geq a \int c u^+ - \int_\Omega |h| u^+ - \mu \int_\Omega \beta(u^+) |Du^+|^2 u^+ - \int_{\Omega} b(x) |Du^+| u^+, \\
        &\geq a \int c u^+ - K  \|h\|_{L^p(\Omega)} - \mu \beta(K) K \|Du^+\|_{L^2(\Omega)}^2 - K \|b\|_{L^2(\Omega)} \|Du^+\|_{L^2(\Omega)}\\
        &\geq a \int cu^+ - K \|h\|_{L^p(\Omega)} - \mu \beta(K) C - \|b\|_{L^q(\Omega)} C,
    \end{aligned}
\end{equation*}
which, due to $u^+$ being uniformly bounded, yields a contradiction for large enough $a$.

\textit{\textbf{Case 2:} $u$ is nonnegative.}

It is essentially the same idea as Case 1, just that instead of $u^+$ as a test function, we use $u$.

\textit{\textbf{Case 3:} $u$ is nonpositive.}

Reasoning as in the Case 1, let $u^-:= \max \{0,-u\}$ be a test function in \eqref{33},

\begin{equation*}
    \int_{\Omega} |Du^-|^2 = \int_\Omega \big [ \lambda c(x) u + b(x) |Du| + \mu \beta(u) |Du|^2 + h(x) + a c(x) \big ] u^-.
\end{equation*}

Next, we use Theorem \ref{Propimportante} to obtain $C_1 >0$ such that $\|u^-\|_{L^{\infty}(\Omega)} \leq C_1$, and Poincaré inequality to obtain
$$C:=C_p C_1 |\Omega| \geq C_p \|u^-\|_{L^2(\Omega)}^2 \geq \int_\Omega |Du^-|^2.$$
Now, by similar reasoning as in Case 1,  it follows that
\begin{equation*}
\begin{aligned}
C &\geq a \int_\Omega c u^- + \int_\Omega h u^-  + \mu \int_\Omega \beta(u) |Du^-|^2 u^- + \int_\Omega b(x) |Du^-| u^- + \lambda \int_\Omega c (u^-)^2, \\
&\geq a\int_{\Omega} c u^- - \int_\Omega |h| u^- - \mu \beta(C_1) C_1\int_{\Omega} |Du^-|^2 - C_1\int_{\Omega} b(x) |Du^-| ,\\
&\geq a \int_{\Omega} c u^- - \|h\|_{L^p(\Omega)} C_1 - \mu \beta(C_1) C_1 - \| b \|_{L^2(\Omega)} C,
\end{aligned} \end{equation*}
which, since $u^-$ is uniformly bounded, yields a contradiction for large enough $a$, finishing the proof.
\end{proof}

\subsection{Proof of Theorem \ref{multpositivas}}
We start by proving the existence of a continuum.
 
 Let $\gamma$ be the solution provided by Theorem~\ref{teo_existenciasolfuerte} to the problem  
\[
\begin{aligned}
    -\Delta \gamma - b(x) |D\gamma| &= |\beta(\gamma)|\, |D\gamma|^2 + h(x) + 1 && \text{in } \Omega,\\
    \gamma &= 1 && \text{on } \partial\Omega.
\end{aligned}
\]
It is easy to verify that $\gamma$ is a supersolution of \eqref{P0prima}, and that $u_0 \ll \gamma$.  
Moreover, by Lemma~\ref{strict}, there exists a \emph{strict subsolution} $\alpha$ of \eqref{P0prima} such that $\alpha \leq \gamma$. Since $\alpha \leq 0$ on $\partial \Omega$, we conclude that $\alpha \ll \gamma$.  
Hence, by Theorem~\ref{thm:21coster}, there exists $R>0$ such that  
\[
\deg(I - \mathcal{M}_0, \mathcal{S}) = 1,
\quad \text{where} \quad
\mathcal{S} := \{\, u \in C_0^1(\overline{\Omega}) \mid \alpha \ll u \ll \gamma \,, \, \|u\|_{C^1(\Omega)} \leq R\,\}.
\]
Now, by Remark \ref{subsuperleq}, we know that $u_0$ is the unique solution of \eqref{P0prima}.  
Therefore, as $R \to 0$, we obtain  
$
i(I - \mathcal{M}_0, u_0) = 1.
$
Finally, by {\cite[Theorem~3.2]{rabinowitz_global_1971}}, we conclude the existence of a continuum of solutions satisfying the announced properties.

With these results, we can address the proof of the multiplicity result obtained for \eqref{Plambdaprima}.
For the rest of the proof, we proceed in several steps.

\textbf{Step 1:} \textit{Every nonnegative supersolution of \eqref{Plambdaprima} satisfies $u \gg u_0.$}

If $u$ is a nonnegative supersolution of \eqref{Plambdaprima}, then $u$ is a supersolution of \eqref{P0prima} (since $cu \gneqq 0$). By Remark \ref{subsuperleq} we deduce that $u \geq u_0$ and hence $u$ is not a solution of \eqref{P0prima}. Now, let $w:= u - u_0$. Note that $w\geq 0$ and it satisfies
\begin{equation*}
    \begin{aligned}
        -\Delta w &\geq \lambda c(x) u + b(x)|Dw| + \mu\beta(u)|Du|^2 - \mu \beta(u_0)|Du_0|^2 \quad &&\text{in } \Omega,\\
        w &\geq 0 &&\text{on } \partial \Omega. 
    \end{aligned}
\end{equation*}
Now, using that $\beta$ is increasing and that,
\begin{equation*}
\begin{aligned}
    \mu ( \beta(u)|Du|^2 - \beta(u_0) |Du_0|^2) &= \underbrace{\mu (\beta(u) - \beta(u_0))|Du|^2}_{\geq 0} + \underbrace{\mu \beta(u_0) ( |Du|^2 - |Du_0|^2)}_{= \mu \beta(u_0) (|Du|+|Du_0|)(|Du|-|Du_0|)} \\
    &\geq \mu \beta(u_0) (|Du_0| + |Du|) |Dw|
    \end{aligned}
\end{equation*}
and defining $\tilde b(x) = b(x) + \mu \beta(u_0) (|Du| + |Du_0|)$ we obtain that
\begin{equation*}
    \begin{aligned}
        -\Delta w - \tilde{b}(x) |Dw|&\geq \lambda c(x) u \geq 0\quad &&\text{in } \Omega,\\
        w &\geq 0 &&\text{on } \partial \Omega. 
    \end{aligned}
\end{equation*}
Since $w \geq 0$, by the strong maximum principle {\cite[Theorem 3.27]{caffarelli_interior_1989}}, we deduce that either $w \gg 0$ or $w \equiv 0$. If $w \equiv 0$, then $u = u_0$ is a solution of \eqref{P0prima}, which contradicts what we deduced before. Therefore $w \gg 0$, i.e., $u \gg u_0$.

\textbf{Step 2:} \textit{Problem \eqref{Plambdaprima} has no nonnegative solution for $\lambda$ large.}

Let $\varphi_1>0$ be the first eigenfunction of
\begin{equation*}
    -\Delta \varphi_1 = \gamma_1 c(x) \varphi_1, \quad \varphi_1 \in H_0^1(\Omega)
.\end{equation*}
If \eqref{Plambdaprima} has a nonnegative solution, multiplying \eqref{Plambdaprima} by $\varphi_1$ and integrating we obtain
\begin{equation*}
    \begin{aligned}
        \gamma_1 \int_\Omega c(x) u \varphi_1 &= - \int_\Omega \Delta u \,\varphi_1 \\
        &= \lambda \int_\Omega c(x) u \varphi_1 + \int_\Omega \mu \beta(u) |Du|^2 \varphi_1 + \int_\Omega b(x)|Du| \varphi_1 + \int_\Omega h(x) \varphi_1,
    \end{aligned}
\end{equation*}
thus, for $\lambda > \gamma_1$, as $u \geq u_0$, we have
\begin{equation*}
    \begin{aligned}
        0 &\geq (\lambda - \gamma_1) \int_\Omega c(x) u \varphi_1 + \int_{\Omega} b(x) |Du| \varphi_1 + \int_{\Omega} \mu \beta(u) |Du|^2 \varphi_1 + \int_{\Omega} h(x) \varphi_1, \\
        &\geq (\lambda - \gamma_1) \int c(x) u_0 \varphi_1+ \int_{\Omega} b(x) |Du| \varphi_1 + \int_{\Omega} \mu \beta(u_0) |Du|^2 \varphi_1 + \int_{\Omega} h(x) \varphi_1 
    \end{aligned}
\end{equation*}
which yields a contradiction for $\lambda $ large enough.

\textbf{Step 3:} \textit{Define $\overline{\lambda} = \sup \{ \lambda \, | \,$ \eqref{Plambdaprima} has a solution $u_\lambda\geq0 \}$, then $\overline{\lambda} \in[0, + \infty)$ and, for all $\lambda > \overline{\lambda},$ \eqref{Plambdaprima} has no nonnegative solution. Moreover, $\overline{\lambda} >0$.}

The first part is direct from the definition of $\overline{\lambda}$ and \textbf{Step 2}. Of course, $\overline{\lambda} \geq0$ since $u_0$ is a positive solution of \eqref{P0prima}. 

Next, to prove the positivity of $\overline{\lambda}$, we fix $\Lambda_0 > 0$ such that, for every $\lambda \in (0, \Lambda_0)$, problem \eqref{Plambdaprima} admits a nontrivial solution.  
Now, let $u_c$ denote the strong solution of  
\begin{equation*}
    \begin{aligned}
        -\Delta u_c - b(x) |Du_c| &= \lambda c(x) |u_c| + \mu |\beta(u_c)| |Du_c|^2 + h(x) &&\text{in } \Omega, \\
        u_c &= 0 &&\text{on } \partial \Omega,
    \end{aligned}
\end{equation*}
whose existence follows from Theorem~\ref{teo_existenciasolfuerte} under a suitable smallness condition on $\lambda $ and $\mu$, up to diminishing $\Lambda_0$ and $\delta^\star$ if necessary.  We emphasize that replacing $\lambda c(x)u$ with $\lambda c(x)|u_c|$ does not affect the proof of that theorem, 
which is why it remains applicable in this particular setting.

Since 
\[
\mu |\beta(u_c)| |Du_c|^2 \geq \mu \beta(u_c) |Du_c|^2, \quad \lambda c(x) |u_c|\geq0,  
\quad \text{and} \quad
\lambda c(x) |u_c| \geq \lambda c(x) u_c,
\]
it follows that $u_c$ is a supersolution of \eqref{Plambdaprima}, and  also a supersolution of \eqref{P0prima}.  

Furthermore, because $0 \leq\lambda c(x) u_0  $, the function $u_0$ is a subsolution of \eqref{Plambdaprima}.  
Since $u_0$ is also a subsolution of \eqref{P0prima}, Remark~\ref{subsuperleq} yields
$
u_0 \leq u_c $ in $ \Omega.
$
Then, by Theorem~\ref{thm:21coster}, there exists a solution $u_\lambda$ of \eqref{Plambdaprima} such that
$
u_0 \leq u_\lambda \leq u_c.
$
Hence, since $u_0 \geq 0$, we conclude that $u_\lambda$ is a nonnegative solution of \eqref{Plambdaprima}. Therefore, because such a solution exists for all $\lambda \in (0, \Lambda_0)$, it follows that $\overline{\lambda} > 0$.

\textbf{Step 4:} \textit{For all $0 < \lambda< \overline{\lambda}$, \eqref{Plambdaprima} has well ordered strict lower and upper solutions.}
Note that $u_0$ is a subsolution of \eqref{Plambdaprima} which is not a solution, since
\begin{equation*}
\begin{aligned}
-\Delta u_0 &= b(x)|Du_0| + \mu \beta(u_0)|Du_0|^2 + h(x) \\&\leq \lambda c(x)u_0 + b(x) |Du_0| + \mu \beta(u_0) |Du_0|^2 + h(x) .
\end{aligned}
\end{equation*}

Now, by definition of $\overline{\lambda}$, we can find 
$\tilde{\lambda} \in (\lambda, \overline{\lambda})$ and a nonnegative solution $u_{\tilde{\lambda}}$ of $(P_{\tilde{\lambda}})$. Then, by a similar argument as for $u_0$ being a subsolution, $u_{\tilde{\lambda}}$ is a supersolution of \eqref{Plambdaprima}. Moreover, it satisfies $u_{\tilde{\lambda}} \gg u_0$ by \textbf{Step 1}.

Let $u$ be a solution of \eqref{Plambdaprima}. Setting $w = u_{\tilde{\lambda}} - u$ and reasoning as in \textbf{Step 1}, we obtain that $w$ satisfies
\begin{equation*}
    \begin{aligned}
        -\Delta w - \tilde{b}(x) |Dw| &\geq \lambda c(x) w \geq 0 \quad &&\text{in } \Omega,\\
        w &\geq 0 &&\text{on } \partial \Omega,
    \end{aligned}
\end{equation*}
where $\tilde{b}(x) = b(x) + \mu \beta(u)( |Du_{\tilde{\lambda}}| + |Du|)$. This, the maximum principle (see {\cite[Theorem 3.27]{troianiello_elliptic_1987}}) implies that $w \gg 0$, hence $u_{\tilde{\lambda}} \gg u$, proving the strictness of $u_{\tilde{\lambda}}$. Analogously, we prove that $u_0$ is a strict subsolution of \eqref{Plambdaprima}.

\textbf{Step 5:} \textit{For all $\lambda \in (0, \overline{\lambda})$, \eqref{Plambdaprima} has at least two positive solutions with $u_0 \ll u_{\lambda,1} \ll u_{\lambda,2}$}.

By \textbf{Step 4}, Theorem \ref{thm:21coster} and Remark \ref{remark22}, we have some $R_0> 0$ such that $\deg(I - \mathcal{M}_\lambda, \mathcal{S}) = 1$ with
$$\mathcal{S} = \{ u \in C_0^1(\overline{\Omega}) \, | \, u_0 \ll u \ll u_{\tilde{\lambda}}, \, \|u\|_{C^1(\Omega)} < R_0 \} $$
and we have the existence of a first solution $u_{\lambda,1}$ of \eqref{Plambdaprima} such that $u_0 \leq u_{\lambda,1} \leq u_{\tilde{\lambda}}.$ As in the proof of {\cite[Theorem 1.4]{de_coster_multiplicity_2017}}, let us choose $u_{\lambda,1}$ as the minimal solution between $u_0$ and $u_{\tilde{\lambda}}$.

Now, using Lemma \ref{36}, there exists $A_1>0$ large enough such that \eqref{33} has no solution for $a\geq A_1$. By Theorem \ref{apriorilaplaciano} and Lemma \ref{Nagumo}, there exists $R_1 > R_0 >0$ such that, for any $a \in [0, A_1]$, every solution of \eqref{33} with $u\geq u_0$ satisfies $\|u\|_{C^1(\overline{\Omega})} < R_1$. Therefore, by the homotopy invariance of the topological degree, we have
\begin{center}
$\deg(I- \mathcal{M}_{\lambda}, \mathcal{D}) = \deg(I - \mathcal{M}_{\lambda} - \mathcal{L}^{-1}(A_1c), \mathcal{D}), $
where
$\mathcal{D} := \{ u \in C_0^1(\overline\Omega) \, \colon \, u_0 \ll u, \|u\|_{C^1(\Omega)} < R_1 \}.$
\end{center}
In the case of $a=A_1$, we know that \eqref{33} has no solution, hence $\deg(I -\mathcal{M_\lambda}-\mathcal{L}^{-1}(A_1 c), \mathcal{D}) = 0$. We then use this to apply the excision property of the degree, obtaining
$$\deg(I-\mathcal{M}_{\lambda}, \mathcal{D}\setminus \mathcal{S}) = \deg(I-\mathcal{M}_{\lambda}, \mathcal{D}) - \deg(I-\mathcal{M}_{\lambda},\mathcal{S}) = 0-1=-1$$
This proves the existence of a second solution $u_{\lambda,2}$ of \eqref{Plambdaprima} with $u_{\lambda,2} \gg u_0$ (since $u_{\lambda,2} \in \mathcal{D}$). As $u_{\lambda,1}$ is the minimal solution between $u_0$ and $u_{\tilde{\lambda}}$, we have $u_{\lambda,1} \lneqq u_{\lambda,2}$ as otherwise, by Theorem \ref{thm:21coster}, we would have a solution $u$ with $u_0 \leq u \leq \min \{ u_{\lambda,1}, u_{\lambda,2}, u _{\tilde{u}} \}$ that contradicts the minimality of $u_{\lambda,1}$. We continue by the same reasoning done in \textbf{Step 1} and \textbf{Step 4} to conclude that $u_{\lambda,1} \ll u_{\lambda,2}$.

\textbf{Step 6:} \textit{For $\lambda_1 < \lambda_2$, we have $u_{\lambda_1, 1} \ll u_{\lambda_2,1}$.}

Recall $u_{\lambda,1}$ is the minimal solution above $u_0$ and, just as in \textbf{Step 4}, $u_{\lambda_2,1}$ is a strict supersolution of $(P_{\lambda_1})$ with $u_{\lambda_2,1} \geq u_0$. We then deduce that $u_{\lambda_1, 1} \ll u_{\lambda_2,1}$.

\textbf{Step 7:} \textit{Problem $(P_{\overline{\lambda}})$ has at least one solution.}

Let $\{\lambda_n\}_{n \in \mathbb{N}} \subset (0, \overline{\lambda})$ be a sequence such that $\lambda_n \to \overline{\lambda}$ and $\{u_n\}_{n \in \mathbb{N}} \subset W^{2,p}(\Omega)$ be a sequence of corresponding solutions of $(P_{\lambda_n})$. By Theorem \ref{apriorilaplaciano}, there exists a constant $M>0$ such that, for all $n \in \mathbb{N}$, $\|u_n\|_{L^{\infty}(\Omega)}< M$ and, by Lemma \ref{Nagumo}, we have that there exists an $R>0$ such that, for all $n \in \mathbb{N}$, $\|u_n\|_{W^{2,p}(\Omega)} < R$. Hence, since $W^{2,p}_0(\Omega) \subset \subset C_0^1(\overline{\Omega})$ we have, up to a subsequence, that $u_n \to u$ in $C_0^1(\overline{\Omega})$. From this strong convergence and Lemma \ref{lemma:51}, we observe that
\begin{equation*}
\begin{aligned}
- \Delta u  &= \overline{\lambda}c(x)u + b(x) |Du| + \mu \beta(u)|Du|^2 + h(x)\quad &&\text{in } \Omega,\\
u&=0 &&\text{on } \partial\Omega,
\end{aligned}\end{equation*}
i.e., $u \in W^{2,p}(\Omega)$ is a solution of $(P_{\overline{\lambda}})$.

\textbf{Step 8:} Behavior of the solutions for $\lambda \to 0$.

Let $\{\lambda_n\}_{n \in \mathbb{N}} \subset (0, \overline{\lambda})$ be a decreasing sequence such that $\lambda_n \to 0$. By \textbf{Step 5} we know that the corresponding solutions $u_{\lambda_n,1}$ satisfy that $u_0 \ll u_{\lambda_n,1}$ and by \textbf{Step 6} we have that $\{u_{\lambda_n,1}\}_{n \in \mathbb{N}}$ is also a decreasing sequence. Arguing as in \textbf{Step 7}, we have that, up to a subsequence, $u_{\lambda_n,1} \to u$ in $C_0^1(\overline{\Omega})$ with $u$ solution of \eqref{P0prima}. By the uniqueness of the solution of \eqref{P0prima}, we deduce that $u=u_0$.

Now, for the behavior of $u_{\lambda,2}$, let us consider the sequence of solutions $\{u_{\lambda_n,2}\}_{n \in \mathbb{N}}$. If they were bounded, then as before, we would have, up to a subsequence, that $u_{\lambda_n,2} \to u$ in $C_0^1(\overline{\Omega})$ with $u$ being a solution of \eqref{P0prima}. By \textbf{Step 5} and the facts that $u_{\lambda_n,2} \not \in \mathcal{S}$, $u_{\lambda_n,2} \gg u_{\lambda_n,1}$ and $u_{\lambda_n,1} \to u_0$, we know that $\max \{u_{\lambda_n,2} - u_0 \} > \varepsilon$. This would imply that $u \not = u_0$ and $u$ solution of \eqref{P0prima}, which contradicts the uniqueness of the solution of \eqref{P0prima} as before. Hence, they cannot be bounded, i.e.,
$$\max u_{\lambda_n,2} \xrightarrow[]{n \to +\infty} + \infty,$$
concluding the proof. \qedsymbol

\bibliographystyle{amsplain}
\bibliography{references} 
\end{document}